\documentclass[11pt,reqno]{amsart}
\usepackage{fullpage}
\usepackage{amsfonts}
\usepackage{amssymb}
\usepackage{times}
\usepackage{graphicx}
\usepackage{amsmath,bm}
\usepackage{tikz-cd}
\usepackage{tikz}
\usepackage{appendix}
\usepackage{color}
\usepackage{mathrsfs}
\usepackage{geometry}
\newtheorem{thm}{Theorem}[section]
\newtheorem{cor}[thm]{Corollary}
\newtheorem{lem}[thm]{Lemma}
\newtheorem{prop}[thm]{Proposition}

\theoremstyle{definition}
\newtheorem{ex}[thm]{Example}
\newtheorem{nota}[thm]{Notation}
\newtheorem{defn}[thm]{Definition}

\theoremstyle{remark}
\newtheorem{rem}[thm]{Remark}
\begin{document}

\title[Short version of title]{smoothing of limit linear series on curves and metrized complexes of pseudocompact type}
\author{Xiang He}
\maketitle

\begin{abstract}
We investigate the connection between Osserman limit series \cite{osserman2014limit} (on curves of pseudocompact type) and Amini-Baker limit linear series \cite{amini2015linear} (on metrized complexes with corresponding underlying curve) via a notion of pre-limit linear series on curves of the same type. Then, applying the smoothing theorems of Osserman limit linear series, we deduce that, fixing certain metrized complexes, or for certain types of Amini-Baker limit linear series, the smoothability is equivalent to a certain ``weak glueing condition". Also for arbitrary metrized complexes of pseudocompact type the weak glueing condition (when it applies) is necessary for smoothability. As an application we confirm the lifting property of specific divisors on the metric graph associated to a certain regular smoothing family, and give a new proof of \cite[Theorem 1.1]{cartwright2014lifting} for vertex avoiding divisors, and generalize loc.cit. for divisors of rank one in the sense that, for the metric graph, there could be at most three edges (instead of two) between any pair of adjacent vertices. 

%We recall two notions of limit linear series by Amini and Baker \cite{amini2015linear} (on metrized complexes) and by Brian Osserman \cite{osserman2014limit} (on curves of pseudocompact type), and define pre-limit linear series on curves of the same type. We first prove that an (Amini-Baker) limit linear series on a metrized complex contains the same data as a pre-limit linear series on the underlying curve. Then, applying the smoothing theorems of Osserman limit linear series in \cite{osserman2014dimension} and \cite{osserman2014limit} as well as its connection with pre-limit linear series, we deduce that, for certain metrized complexes, an Amini-Baker limit linear series is smoothable if and only if it satisfies the so called ``weak glueing condition". Also for arbitrary metrized complexes of pseudocompact type we prove that a the weak glueing condition is necessary for smoothableness, and a general limit linear series (whose restriction to the underlying graph is rational) is smoothable. As an application we confirm the lifting property of specific divisors on certain metric graphs associated to regular smoothing families, generalizing \cite[Conjecture 1.5]{cools2012tropical} for divisors of rank one in the sense that, for the metric graph, there can be at most three edges between any pair of adjacent vertices.
\end{abstract}

%\tableofcontents

\section{Introduction}
The theory of limit linear series has been developed by Eisenbud and Harris in \cite{eisenbud1986limit} for handling the degeneration of linear series on smooth curves as the curves degenerate to reducible curves (of compact type). It has been applied to prove results involving moduli space of curves, such as the Brill-Noether theorem (\cite{griffiths1980variety}), the Gieseker-Petri theorem (\cite{gieseker1982stable}), and that moduli spaces of curves of sufficiently high genus are of general type (\cite{harris1982kodaira} and \cite{eisenbud1987kodaira}), etc.

A complete generalization of Eisenbud-Harris theory has remained open. Earlier approaches can be found in papers of Eduardo Esteves such as \cite{esteves1998linear} and \cite{esteves2002limit}. Recently, Amini and Baker \cite{amini2015linear} introduced a notion of metrized complexes, which is roughly speaking a finite metric graph $\Gamma$ together with a collection of marked curves $C_v$, one for each vertex $v$, such that the marked points of $C_v$ is in bijection with the edges of $\Gamma$ incident on $v$. This can be considered as an enrichment of both metric graphs and nodal curves. The concept of limit linear series on a metrized complex is proposed in loc.cit. as well as the specialization map. While the Amini-Baker limit linear series satisfies the specialization theorem, it is unclear how to prove a general theorem of their smoothing behaviors. Related results can be found in \cite{luo2014smoothing}, where a sufficient and necessary condition for the smoothability (with respect to certain base field) of a (saturated) limit linear series of rank one is given. 

On the other hand, Brian Osserman developed in \cite{osserman2014limit} a notion of limit linear series on curves possibly not of compact type, which is a generalization of Eisenbud and Harris limit linear series. He also proved the specialization theorem, as well as a smoothing theorem for curves of pseudocompact type (see below for definition), which states that a limit linear series is smoothable if the moduli space is of expected dimension at the corresponding point. This improves the smoothing theorem in the compact-type case in the sense that it applies for possibly non-refined limit linear series.

In the present paper we investigate the smoothing of Amini-Baker limit linear series by studying their connection with Osserman limit linear series.

%for the two kinds of limit linear series mentioned above. It is natural to first consider the connection between them.%limit linear series on metrized complexes and limit linear series on curves not of compact type. 

Let $X_0$ be a curve of {pseudocompact type}, which is a curve whose dual graph $G$ is obtained from a tree, denoted by $\overline G$, by adding edges between adjacent vertices. A {chain structure} $\bm n$ on $G$ is roughly a integer-valued length function on $E(G)$. This induces a metric graph $\Gamma$ as well as a metrized complex $\mathfrak C_{X_0,\bm n}$ with underlying graph $\Gamma$ (see \S 2 for details). Given a chain structure $\bm n$, let $\widetilde X_0$ be the curve obtained from $X_0$ by inserting a chain of $\bm n(e)-1$ projective lines at the node of $e$ for all $e\in E(G)$. Let $\widetilde G$ be the dual graph of $\widetilde X_0$. An Osserman limit linear series (which we also refer to as a limit linear series on $(X_0,\bm n)$) then consists of a certain line bundle $\mathscr L$ on $\widetilde X_0$ and a collection of linear systems on each component of $X_0$ that satisfies certain (multi)vanishing conditions, and an Amini-Baker limit linear series (which we also refer to as a limit linear series on $\mathfrak C_{X_0,\bm n}$) consists of a divisor on $\mathfrak C_{X_0,\bm n}$ with a collection of linear spaces of rational functions on each component of $X_0$ satisfying certain rank conditions. The multidegree $w_0$ of a limit linear series on $(X_0,\bm n)$ is the multidegree of $\mathscr L$ on $V(\widetilde G)$, which induces a divisor on $\Gamma$ as we identify $\Gamma$ with the metric graph obtained from $\widetilde G$ by assigning length $1$ to every edge. We also call a limit linear series on $\mathfrak C_{X_0,\bm n}$ of multidegree $w_0$ if its underlying divisor on $\Gamma$ is (up to linear equivalence) induced by $w_0$.

In the following we fix $w_0$ the multidegree of a limit linear series on $(X_0,\bm n)$ and $\mathfrak C_{X_0,\bm n}$. Note that, by the definition of Osserman limit linear series, $w_0$ is assumed to be ``admissible", namely, when restricting to each component of $\Gamma\backslash V(G)$, the corresponding divisor $D$ of $w_0$ is effective and of at most degree one, see also Definition \ref{admissible multidegree}.

In \cite{osserman2017limit} the author constructed a map $\mathfrak F$ from the set of limit linear series on $(X_0,\bm n)$ to the set of limit linear series on $\mathfrak C_{X_0,\bm n}$. It is also proved in loc.cit. that $\mathfrak F$ can be defined over the set of pre-limit linear series (a weakened version of Osserman limit linear series), in this case we show in Section 3 that $\mathfrak F$ is a bijection. This essentially says that a limit linear series on $\mathfrak C_{X_0,\bm n}$  carries the same data as a pre-limit linear series on $(X_0,\bm n)$. We also give a necessary condition for a pre-limit linear series (hence for a limit linear series on $\mathfrak C_{X_0,\bm n}$) to be lifted to a limit linear series on $(X_0,\bm n)$, which is called the weak glueing condition. 

In Section $4$ we consider the smoothing of limit linear series on $\mathfrak C_{X_0,\bm n}$. We show in Theorem \ref{smoothness implies weak glueing} that a necessary condition is the weak glueing condition. On the other hand, it is proved in \cite{osserman2014dimension} that for a special case of $(X_0,\bm n)$ the moduli space of limit linear series (again of multidegree $w_0$) is of expected dimension, hence every limit linear series on $(X_0,\bm n)$ is smoothable. We show in this case that the weak glueing condition is sufficient for a pre-limit linear series to be lifted to a limit linear series, and use the ``equivalence" between pre-limit linear series on $(X_0,\bm n)$ and limit linear series on $\mathfrak C_{X_0,\bm n}$ to give a smoothing theorem for the latter (see Theorem \ref{general edge length smoothing} for details):

%We then look at a special case of $(X_0,\bm n)$ such that the weak glueing condition is sufficient for a pre-limit linear series to lift to a limit linear series, and that the moduli space of limit linear series (again of multidegree $w_0$) is of expected dimension. The counting of dimension is carried out in \cite{osserman2014dimension}. We then use Osserman's smoothing theorem as mentioned above along with the ``equivalence" between pre-limit linear series on $(X_0,\bm n)$ and limit linear series on $\mathfrak C_{X_0,\bm n}$ to give a smoothing theorem for the latter (see Corollary \ref{general edge length smoothing} for details):

\begin{thm}\label{introduction smoothing theorem}
Suppose we have $(X_0,\bm n)$ such that the induced metric graph $\Gamma$ has few edges and general edge lengths, and that the components of $X_0$ are strongly Brill-Noether general, then a limit linear series on $\mathfrak C_{X_0,\bm n}$ is smoothable if and only if it satisfies the weak glueing condition.
\end{thm}

Additionally, for arbitrary $\bm n$ and $X_0$ with strongly Brill-Noether general components we consider a family of $w_0$ such that the induced divisor $D$ on $\Gamma$ is ``randomly distributed" on $\Gamma\backslash V(G)$, as in Theorem \ref{smoothing general limit grd}. More precisely, for any edge $e\in E(\overline G)$, let $e_1^\circ,...,e_m^\circ$ be the components of $\Gamma\backslash V(G)$ corresponding to edges $e_1,...,e_m\in E(G)$ that lies over $e$. Given a certain direction on $G$, the divisor $D|_{e_i^\circ}$ gives an integer $x_i$ in $[0,\bm n(e_i)-1]$, and we consider $w_0$ such that $x_1,...,x_m$ are distinct modulo $\mathrm{gcd}(\bm n(e_1),...,\bm n(e_m))$.
We show that a pre-limit linear series automatically lifts to a limit linear series on $(X_0,\bm n)$, and the dimension of the moduli space of limit linear series on $(X_0,\bm n)$ is as expected. In this case any limit linear series on $\mathfrak C_{X_0,\bm n}$ is smoothable.

In section $5$ we consider
%we apply the smoothing results we have for limit linear series on $(X_0,\bm n)$, where $(X_0,\bm n)$ is as in Theorem \ref{introduction smoothing theorem}, to 
the problem of lifting divisors on the metric graph $\Gamma$ to the generic fiber $X_\eta$ of any regular smoothing family (Definition \ref{smoothing family}) $X$ with special fiber $\widetilde X_0$ with rational components and dual graph $\widetilde G$. When $\Gamma$ is a generic chain of loops, it is proved by reducing to the so-called vertex avoiding divisors that every rational divisor on $\Gamma$ is liftable (see \cite{cartwright2014lifting} for details). For $(X_0,\bm n)$ as in Theorem \ref{introduction smoothing theorem}, again since the weak glueing condition is sufficient for a pre-limit linear series to be lifted to a limit linear series on $(X_0,\bm n)$, and the dimension counting shows that every limit linear series on $(X_0,\bm n)$ is smoothable, we are able to prove the following theorem, which gives an alternate approach of lifting (rational) vertex avoiding divisors on a generic chain of loops, by lifting the divisor on $\Gamma$ to a pre-limit linear series on $(X_0,\bm n)$ that satisfies the weak glueing condition.

%for specific $(X_0,\bm n)$, we lift the divisor on $\Gamma$ to a pre-limit linear series on $(X_0,\bm n)$, or equivalently a limit linear series on $\mathfrak C_{X_0,\bm n}$, that satisfies the weak glueing condition, which can be lifted further to a smoothable limit linear series on $(X_0,\bm n)$, as the dimension counting shows that every limit linear series on $(X_0,\bm n)$ is smoothable. Now the compatibility between the specializations of linear series and divisors on $X_\eta$ gives:

\begin{thm}
Let $(X_0,\bm n)$ be as in Theorem \ref{introduction smoothing theorem} and $X$ be as above. Suppose further that $X_0$ only has rational components, and that $\overline G$ is a chain. Then every rational divisor on $\Gamma$ of rank less than or equal to $1$ lifts to a divisor on $X_\eta$ of the same rank. In addition, if $\Gamma$ is a generic chain of loops then every rational vertex avoiding divisor  lifts to a divisor on $X_\eta$ with the same rank.
\end{thm}

See Theorem \ref{lifting divisors chain of double loops} and Theorem \ref{lifting vertex avoiding divisors} for more precise statements. Note that, ignoring the field condition (see below), the first part of the theorem confirms \cite[Conjecture 1.5]{cools2012tropical} and \cite[Theorem 1.1]{cartwright2014lifting} for divisors of rank one, and generalizes the conditions in loc.cit. for the graph $\Gamma$ in the sense that there could be at most three edges (instead of two) between any pair of adjacent vertices.
 %and the second part gives an alternate proof of lifting vertex avoiding divisors on a generic chain of loops, while the lifting of arbitrary rational divisor on $\Gamma$ is solved in \cite{cartwright2014lifting}. We refer to loc.cit. for the notions of a generic chain of loops and vertex avoiding divisors. 

\subsection{Conventions and Notations.}
All curves we consider are proper, (geometrically) reduced and connected, and at worst nodal. All nodal curves are split. All irreducible components of a nodal curve are smooth.

In the sequel we let $R$ be a complete discrete valuation ring with valuation $\mathrm{val}\colon R\rightarrow \mathbb Z_{\geq 0}$ and residue field $\kappa$. Let $K$ be the fraction field of $R$, which is a non-archimedean field with the induced norm $|x|=\exp(-\mathrm{val}(x))$. Let $\widetilde K$ be the completion of the algebraic closure of $K$ and $\widetilde R$ the valuation ring of $\widetilde K$. Note that $\widetilde K$ is still algebraically closed, the valuation/norm extends uniquely to $\widetilde K$, and 
$\widetilde K$ has residue field $\kappa$ if $\kappa$ is algebraically closed (cf. \cite[\S 1.1]{conrad2008several}).

Let $X_0$ be a curve with dual graph $G$. For $v\in V(G)$ let $Z_v$ be the irreducible component of $X_0$ corresponding to $v$. For $ e'\in E(G)$ that is incident on $v$ let $P_{ e'}$ be the node of $X_0$ corresponding to $ e'$ and $P_{e'}^v$ the preimage (of the normalization map of $X_0$) of $P_{e'}$ in $Z_v$. Let $\overline G$ be the graph obtained from $G$ as follows: for each pair of adjacent vertices $v$ and $v'$ replace all edges connecting $v$ and $v'$ by a single edge. For $ e\in E(\overline G)$ we denote $$\mathcal A_{ e}^v=\cup_{ e'}P_{ e'}^v\ \ \mathrm{and}\ \ \mathcal A_v=\cup_{e\ni v}\mathcal A_{ e}^v$$
where in the first expression $ e' $ runs through all edges of $G$ that lies over $ e$.

For a curve (graph, metric graph, metrized complex) $\mathscr X$ denote $\mathrm{Div}(\mathscr X)$ the space of divisors on $\mathscr X$. If $\mathscr X$ is a curve let $K(\mathscr X)$ be the space of rational functions on $\mathscr X$; if $\mathscr X$ is a graph (metric graph, metrized complex) denote the space of rational functions on $\mathscr X$ by $\mathrm{Rat}(\mathscr X)$. Given a rational function $f$ on $\mathscr X$ denote $\mathrm{div}(f)$ the associated divisor.

Let $C$ be a curve. Let $\mathscr O_C(D)$ be an invertible sheaf on $C$ where $D\in \mathrm{Div}(C)$ and take a nonzero section $s\in H^0(C,\mathscr O_C(D))$. We denote $\mathrm{div}^0(s)$ the effective divisor associated to $s$ that is rationally equivalent to $D$. For a divisor $D'\in\mathrm{Div}(C)$ and $P\in C$ denote $\mathrm{ord}_P(D')$ the coefficient of $P$ in $D'$, denote $\mathrm{ord}^0_P(s)=\mathrm{ord}_P(\mathrm{div}^0(s))$. For a rational function $f\in K(C)$ let $\mathrm{ord}_P(f)$ be the vanishing order of $f$ at $P$. Hence if $s$ is considered as a rational function then we have $\mathrm{ord}_P^0(s)=\mathrm{ord}_P(s)+\mathrm{ord}_P(D)$.  

Let $\Gamma$ be a metric graph with underlying graph $G$. For $e\in E(G)$ and $v\in V(G)$ incident on $e$ and $f\in \mathrm{Rat}(\Gamma)$, let $\mathrm{slp}_{e,v}(f)$ be the outgoing slope of $f$ at $v$ along the tangent direction corresponding to $e$. For $x\in \Gamma$ let $\mathrm{ord}_x(f)$ be the sum of outgoing slopes of $f$ over all tangent directions of $x$.
\\\\
\textbf{Acknowledgements.}
The author would like to thank Brian Osserman for introducing this problem and for helpful conversations, and thank Sam Payne for the idea of the proof of Theorem \ref{lifting vertex avoiding divisors}.

\section{Preliminaries}
We recall some relative notions about Osserman and Amini-Baker limit linear series.
Let $G$ be a graph without loops. Recall that a \textbf{chain structure} on $G$ is a function $\bm n\colon E(G)\rightarrow \mathbb Z_{>0}$. Let $\Gamma$ be the corresponding metric graph with edge length defined by $\bm n$, and $\widetilde G$ the graph obtained from $G$ by inserting $\bm n(e)-1$ vertices between the vertices adjacent to $e$ for all $e\in E(G)$. Then $\Gamma$ can also be obtained from $\widetilde G$ by associating unit edge lengths to all of $E(\widetilde G)$. We use $V(G)$ or $V(\widetilde G)$ instead of $V(\Gamma)$ for different choices of vertex sets of $\Gamma$.

\subsection{Some notions for graph theory.} We recall some concepts about admissible multidegrees (\cite{osserman2014dimension}) on a graph as well as reduced divisors (\cite{luo2011rank}).

\begin{defn}\label{admissible multidegree}
An \textbf{admissible multidegree} % (or a multidegree for short\footnote{Note that in \cite{osserman2014limit} admissible multidegrees and multidegrees have distinct definitions over distinc graphs, but they can be naturally induced by each other, hence there is no essential difference.}) 
 $w=(w_G,\mu)$ of total degree $d$ on $(G,\bm n)$ consists of a function $w_G\colon V(G)\rightarrow \mathbb Z$ together with a tuple $(\mu(e))_{e\in E(G)}\in\Pi_{e\in E(G)}\mathbb Z/\bm n(e)\mathbb Z$ such that $d=\#\{e\in E(G)|\mu(e)\neq 0\}+\sum_{v\in V(G)}w_G(v)$.
\end{defn} 

Correspondingly, we define the notion of (integral) edge-reduced divisors, which is closely related to admissible multidegrees, as follows:

\begin{defn}
A divisor $ D$ on $\Gamma$ is \textbf{edge-reduced} if the restriction of $ D$ on each connected component of $\Gamma\backslash V(G)$ is either empty or an effective divisor of degree one. We say that $ D$ is \textbf{rational} (resp. \textbf{integral}) if $ D$ is supported on rational (resp. integral) points.
\end{defn}

Suppose $G$ is directed, there is a natural bijection $\varphi$ between the set of integral edge-reduced divisors on $\Gamma$ (of degree $d$) and the set of admissible multidegrees on $(G,\bm n)$ (of degree $d$). Precisely, for $D\in\mathrm{Div}(\Gamma)$ integral and edge-reduced, we set $w_G(v)=\deg(D|_v)$ and for $e\in E(G)$ with tail $v$ let $\mu(e)$ be the distance between the point in $D|_{\tilde e^\circ}$ and $v$ if $D|_{\tilde e^\circ}\neq 0$ and $\mu(e)=0$ otherwise, where $\tilde e$ is the edge of $\Gamma$ corresponding to $e$. 
Denote $D_w=\varphi^{-1}(w)$ for any admissible multidegree $w$.% A divisor $D$ on $\Gamma$ said \textbf{of multidegree} $w$ if $D=D_w$.
%We say that an integral edge-reduced divisor on $\Gamma$ is \textbf{of multidegree} $w$ if the corresponding admissible multidegree constructed as above is $w$. 

For each pair of an edge $e$ and an adjacent vertex $v$, let $\sigma(e,v)=1$ if $v$ is the tail of $e$ and $-1$ otherwise. We have the following definition of twisting:

\begin{defn}\label{twist}
Let $w=(w_G,\mu)$ be an admissible multidegree on $(G,\bm n)$ and $v\in V(G)$. For each $e\in E(G)$ incident on $v$ we do the following operation:

(1) If $\mu(e)+\sigma(e,v)=0$ increase $w_G(v')$ by $1$ where $v'$ is the other vertex of $e$;

(2) If $\mu(e)=0$ decrease $w_G(v)$ by $1$; 

(3) Increase $\mu(e)$ by $\sigma(e,v)$.

The resulting admissible multidegree is called the \textbf{twist} of $w$ at $v$. The \textbf{negative twist} of $w$ at $v$ is the admissible multidegree $w'$ such that the twist of $w'$ at $v$ is equal to $w$.
 
%The \textbf{twist} of $w$ at $v$ is a multidegree obtained as follows: for each $e$ incident on v increase $\mu(e)$ by $\sigma(e,v)$, and if the new $\mu(e)$ is zero increase $w_G(v')$ by $1$ where $v'$ is the other vertex of $e$; then decrease $w_G(v)$ by the number of $e$ incident on $v$ for which $\mu(e)$ had been equal to $0$. The \textbf{negative twist} of $w$ at $v$ is the multidegree $w'$ such that the twist of $w'$ at $v$ is $w$.
\end{defn} 

Let $w_0$ be an admissible multidegree. We denote by $G(w_0)$ the directed graph with vertex set consisting of all admissible multidegrees obtained from $w_0$ by sequences of twists, and with an edge from $w$ to $w'$ if $w'$ is obtained from $w$ by twisting at some vertex $v\in V(G)$. Given $w\in V(G(w_0))$ and $v_1,...,v_m\in V(G)$ (not necessarily distinct), let $P(w,v_1,...,v_m)$ denote the path in $V(G(w_0))$ obtained by starting at $w$ and twisting successively at each $v_i$.

In the sense of an integral edge-reduced divisor $D$ on $\Gamma$ twisting at $v$ (where we choose $V(G)$ as the vertex set of $\Gamma$) is just firing all chips at edges incident on $v$ away from $v$ by distance 1 (assume that we put in advance one chip of $D|_v$ at the position of $v$ in $e$ for all $e$ incident on $v$ such that $\mu(e)=0$). Hence the definition of twisting is independent of the direction of $G$, and we get a linearly equivalent divisor after twisting.

\begin{ex}\label{twisting}
Consider a graph $G$ consists of two vertices $v$ and $v'$ connected by three edges $e_1,e_2,$ and $e_3$. Take a chain structure $\bm n$ with $\bm n(e_1)=4$, $\bm n(e_2)=2$ $\bm n(e_3)=3$ and a direction from $v$ to $v'$. Let $w=(w_G,\mu)$ where $w_G(v)=3$, $w_G(v')=0$, $\mu(e_1)=1$, $\mu(e_2)=1$ and $\mu(e_3)=0$. Then $D_w$ is as in the left of the following graphs, with each number represents the coefficient of the corresponding node in $D_w$. 
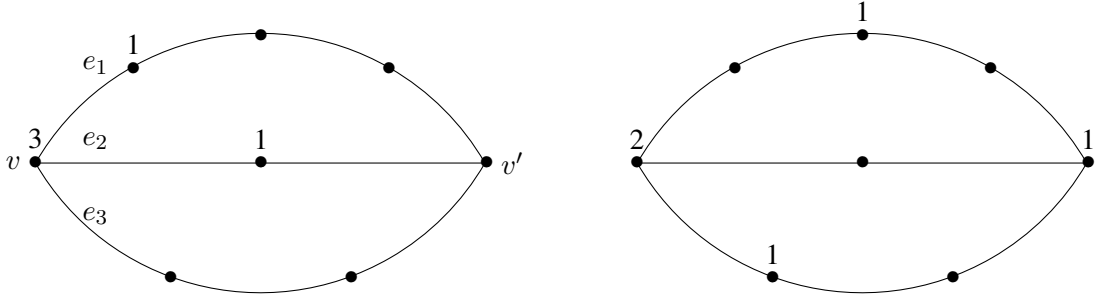
$$
\begin{tikzpicture}
\draw (0,0) -- (6,0);
\draw (0,0) arc (210:330:3.45cm and 3.45cm);
\draw (0,0) arc (-210:-330:3.45cm and 3.45cm);
\draw(0,0) node{$\bullet$};
\draw(3,0) node{$\bullet$};
\draw(6,0) node{$\bullet$};
\draw(3,1.7) node{$\bullet$};
\draw(1.3,1.26) node{$\bullet$};
\draw(4.7,1.26) node{$\bullet$};
\draw(1.8,-1.53) node{$\bullet$};
\draw(4.2,-1.53) node{$\bullet$};
\draw(1.3,1.26)node[circle, fill=black, scale=0.3, label=above:{1}]{};
\draw(3,0)node[circle, fill=black, scale=0.3, label=above:{1}]{};
\draw(0,0)node[circle, fill=black, scale=0.3, label=above:{3}]{};
\draw(0,0)node[circle, fill=black, scale=0.3, label=left:{$v$}]{};
\draw(0.8,1.3)node{$e_1$}{};
\draw(0.8,0.3)node{$e_2$}{};
\draw(0.8,-0.7)node{$e_3$}{};
\draw(6,0)node[circle, fill=black, scale=0.3, label=right:{$v'$}]{};
\draw (8,0) -- (14,0);
\draw (8,0) arc (210:330:3.45cm and 3.45cm);
\draw (8,0) arc (-210:-330:3.45cm and 3.45cm);
\draw(8,0) node{$\bullet$};
\draw(11,0) node{$\bullet$};
\draw(14,0) node{$\bullet$};
\draw(11,1.7) node{$\bullet$};
\draw(9.3,1.26) node{$\bullet$};
\draw(12.7,1.26) node{$\bullet$};
\draw(9.8,-1.53) node{$\bullet$};
\draw(12.2,-1.53) node{$\bullet$};
\draw(11,1.7) node[circle, fill=black, scale=0.3, label=above:{1}]{};
\draw(8,0)node[circle, fill=black, scale=0.3, label=above:{2}]{};
\draw(14,0) node[circle, fill=black, scale=0.3, label=above:{1}]{};
\draw(9.8,-1.53) node[circle, fill=black, scale=0.3, label=above:{1}]{};
\end{tikzpicture} 
$$
After twisting at $v$ we get $w'=(w_G',\mu')$ where $w_G'(v)=2$, $w_G'(v')=1$, $\mu'(e_1)=2$, $\mu'(e_2)=0$ and $\mu'(e_3)=1$. The induced $D_{w'}$ is given in the right graph.
\end{ex}

Let $\overline G$ denote the graph obtained from $G$ by contracting edges between every pair of adjacent vertices into one single edge. We say that $G$ is a \textbf{multitree} if $\overline G$ is a tree. Any graph with two vertices is a multitree. For multitrees we also consider partial twists:

\begin{defn}
Let $G$ be a multitree, and $(e,v)$ a pair of an edge $e$ and an adjacent vertex $v$ of $\overline G$, and $w$ an admissible multidegree on $(G,\bm n)$. The \textbf{twist} of $w$ at $(e,v)$ is the admissible multidegree obtained from $w$ by doing the operations as in Definition \ref{twist} for all edges in $G$ that lies over $e$.
\end{defn}

Again for an integral edge-reduced divisor on $\Gamma$ twisting at $(e,v)$ is just firing the chips at edges over $e$ away from $v$ by distance $1$. Note that the twists are commutative, and $w$ remains the same after twisting at all vertices of $G$. Hence the twists are invertible, as the negative twist at $v$ is the same as the composition of the twists at all $v'\neq v$. In addition, for multitrees twisting $w$ at $(e,v)$ 
is the same as twisting at all vertices $v'$
in the connected component of $\overline G\backslash\{e\}$ that contains $v$. If $v'$ is the other vertex of $e$ then twisting at $(e,v)$ is the inverse of twisting at $(e,v')$.	

\begin{defn}
Let $G$ be a multitree. An admissible multidegree $w$ is \textbf{concentrated} on $v$ if there is an ordering on $V(G)$ starting at $v$, and such that for each subsequent vertex $v'$, we have that $w$ becomes negative in vertex $v'$ after taking the composition of the negative twists at all previous vertices. A tuple $(w_v)_{v}$ of admissible multidegrees is \textbf{tight} if $w_v$ is concentrated at $v$ for all $v\in V(G)$, and for all $e\in E(\overline G)$ incident on vertices $v_1$ and $v_2$, we have that $w_{v_1}$ is obtained from $w_{v_2}$ by twisting $b_{v_1,v_2}$ times at $(e,v_2)$ for some $b_{v_1,v_2}\in \mathbb Z_{\geq 0}$.
\end{defn}

If $G$ has only two vertices $v$ and $v'$ then $w$ being concentrated on $v$ is the same as $w$ being negative at $v'$ after twisting at $(e,v')$. An example of a tight tuple of admissible multidegrees is given by the reduced divisors:

\begin{defn}
Given a vertex $v_0\in V(G)$, a divisor $D$ on $G$ is $v_0$-\textbf{reduced} if: (1) $D$ is effective on $V(G)\backslash \{v_0\}$; (2) for every nonempty subset $S\subset V(G)\backslash\{v_0\}$, there is some $v\in S$ such that $\deg (D|_v)$ (the coefficient of $v$ in $D$) is strictly smaller than the number of edges from $v$ to $V(G)\backslash S$.

Given a point $x_0\in \Gamma$, a divisor $D\in \mathrm{Div}(\Gamma)$ is $x_0$-\textbf{reduced} if: (1) $D$ is effective on $\Gamma\backslash \{x_0\}$; (2) for any closed connected subset $A\subset \Gamma$ there is a point $x\in \partial A$ such that $\deg (D|_x)$ is strictly less than the number of tangent directions of $\Gamma\backslash(A\backslash x)$ at $x$.
\end{defn}
 
For every divisor $D\in \mathrm{Div}(G)$ (resp. $D\in \mathrm{Div}(\Gamma)$) and $v_0\in V(G)$ (resp. $x_0\in \Gamma$) there is a unique divisor $D_{v_0}$ (resp. $D_{x_0}$) such that $D_{v_0}$ (resp. $D_{x_0}$) is linearly equivalent to $D$ and $v_0$-reduced (resp. $x_0$-reduced). One easily checks that a divisor $D\in  \mathrm{Div}(\widetilde G)$ is $v_0$ reduced if and only if it is $v_0$ reduced as a divisor on $\Gamma$.

\begin{prop}\label{reduced concentrated}
Let $G$ be a multitree. Let $w_v^{\mathrm{red}}\in V(G(w_0))$ be the admissible multidegree such that $D_{w_v^{\mathrm{red}}}$ is $v$-reduced. Then $w_v^{\mathrm{red}}$ is the unique admissible multidegree in $V(G(w_0))$ which is concentrated on $v$ and nonnegative on all $v'\neq v$, and $(w_v^{\mathrm{red}})_v$ is a tight tuple. 
\end{prop}
\begin{proof}
The first conclusion follows directly from \cite[Corollary 3.9]{osserman2017limit}. According to Dhar's burning algorithm (\cite[Algorithm 2.5]{luo2011rank}), for $v$ and $v'$ in $V(\overline G)$ adjacent to $e\in E(\overline G)$, we have that $w_{v'}^{\mathrm{red}}$ is obtained from $w_v^{\mathrm{red}}$ by twisting $b$ times at $(e,v)$, where $b$ is the largest number (could be negative) such that the resulting multidegree is effective on $V(\overline G)\backslash \{v'\}$.
\end{proof}

\subsection{Limit linear series on curves of pseudocompact type.}
In this subsection we recall the definition of limit linear series by Osserman in \cite{osserman2014dimension} for curves of pseudocompact type. Note that the notion of limit linear series for more general curves is given in \cite[Definition 2.21]{osserman2014limit}.
\begin{defn}\label{smoothing family} 
We say that $\pi\colon X\rightarrow B$ is a \textbf{smoothing family} if $B$ is the spectrum of a DVR, and:

(1) $\pi$ is flat and proper;

(2) the special fiber $X_0$ of $\pi$ is a (split) nodal curve;

(3) the generic fiber $X_\eta$ is smooth;

(4) $\pi$ admits sections through every component of $X_0$.

If further $X$ is regular we say that $X$ is a \textbf{regular smoothing family}.
\end{defn}

Note that if $B=\mathrm{Spec}(R)$ then by completeness condition (4) in the above definition is satisfied automatically. Let $X_0$ be a curve over $\kappa$ with dual graph $G$. For $v\in V(G)$, let $Z_v$ be the corresponding irreducible component of $X_0$ and $Z_v^c$ the closure of the complement of $Z_v$. 

\begin{defn}\label{enriched structure}
An \textbf{enriched structure} on $X_0$ consists of the data, for each $v\in V(G)$ a line bundle $\mathscr O_v$ on $X_0$, satisfying: 

(1) for any $v\in V(G)$ we have $\mathscr O_v|_{Z_v}\cong \mathscr O_{Z_v}(-(Z_v^c\cap Z_v))$ and $ \mathscr O_v|_{Z_v^c}\cong \mathscr O_{Z_v^c}(Z_v^c\cap Z_v)$; 

(2) $\bigotimes_{v\in V(G)}\mathscr O_v\cong \mathscr O_{X_0}$.
\end{defn}

Note that an enriched structure is always induced by any regular smoothing of $X_0$ (cf. \cite[Proposition 3.10]{osserman2014limit}).

Take a chain structure $\bm n$ on $G$ and an admissible multidegree $w_0$. Let $\widetilde X_0$ be the nodal curve obtained from $X_0$ by, for each $e\in E(G)$, inserting a chain of $\bm n(e)-1$ projective lines at the corresponding node, and $\widetilde G$ the dual graph of $\widetilde X_0$. We say that a divisor on $\widetilde G$ is \textbf{of multidegree} $w_0$ if (considered as a divisor on $\Gamma$) it is equal to $D_{w_0}$, and a line bundle $\mathscr L$ on $\widetilde X_0$ is \textbf{of multidegree} $w_0$ if its associated divisor on $\widetilde G$ is.
%Then admissible multidegrees on $(G,\bm n)$ can be considered as divisors on $\widetilde G$, and multidegrees of line bundles on $\widetilde X_0$ on each component. Take a line bundle $\mathscr L$ on $\widetilde X_0$ of multidegree $w_0$. 
 
Given an enriched structure $(\mathscr O_v)_{v\in V(\widetilde G)}$ on $\widetilde X_0$ and a tuple of admissible multidegrees $(w_v)_{v\in V(G)}$ in $ V(G(w_0))$ such that $w_v$ is concentrated on $v$ and a line bundle $\mathscr L$ on $\widetilde X_0$ of multidegree $w_0$, we get a tuple of line bundles $(\mathscr L_{w_v})_{v\in V(G)}$. Roughly speaking, suppose $w_v$, as a divisor on $\widetilde G$, is obtained from $w_0$ by consecutively fire chips at a set $V\subset V(\widetilde G)$ of vertices, then $\mathscr L_{w_v}$ is obtained from $\mathscr L$ by tensoring with $\mathscr O_v$ for all $v\in V$. Note that $\mathscr L_{w_v}$ is of multidegree $w_v$. See \cite[\S 2]{osserman2014dimension} for details of this construction. 
 
\begin{defn}\label{pseudocompact type}
A curve over $\kappa$ is \textbf{of pseudocompact type} if its dual graph is a multitree.
\end{defn}

Now suppose $X_0$ is a curve of pseudocompact type, and the tuple $(w_v)_v$ is tight. For each pair $(e,v)$ of an edge and an adjacent vertex of $\overline G$, let $(D_i^{e,v})_{i\geq 0}$ be the effective divisors on $Z_v$ defined by setting $D_0^{e,v}=0$ and 
$$ D_{i+1}^{e,v}-D_i^{e,v}=\sum_{\begin{tiny}\begin{array}{cc}e'\ \mathrm{over}\ e\\\sigma(e',v)\mu_v(e')\equiv-i(\mathrm{mod}\bm n(e'))\end{array}\end{tiny}}P_{e'}^v$$
where $\mu_v\colon E(G)\rightarrow \mathbb Z/\bm n(e)\mathbb Z$ is induced by $w_v$.
Intuitively $D_i^{e,v}$ records the chips we lose (in every direction) at vertex $v$ when twisting $i$ times at $(e,v)$ from $w_v$. 

\begin{ex}
In Example \ref{twisting}, let $w_v=w$ and $w_{v'}$ be obtained from $w_v$ by twisting three times at $(e,v)$. It is easy to check that $(w_v,w_{v'})$ is a tight tuple. Straightforward calculation shows that 
$$D_i^{e,v}=0,P_{e_3}^v,P_{e_2}^v+P_{e_3}^v,P_{e_2}^v+P_{e_3}^v,P_{e_1}^v+2P_{e_2}^v+2P_{e_3}^v,...$$
for $i=0,1,2,3,4,...$
\end{ex}

For convenience we call $D_i^{e,v}$ the \textbf{twisting divisors} associated to $(w_v)_v$. In order to define limit linear series, we have
%a line bundle on $\widetilde X_0$ and an enriched structure on $\widetilde X_0$ induce 
the following glueing isomorphism: 

\begin{prop}\label{glueing isomorphism} $\mathrm{(}$\cite[Proposition 2.14]{osserman2014dimension}$\mathrm{)}$
Let $X_0$ be a curve of pseudocompact type and $\widetilde X_0, (w_v)_v, \mathscr L, \mathscr L_{w_v}$ be as above. Denote $\mathscr L^v=\mathscr L_{w_v}|_{Z_v}$. Take vertices $v$ and $v'$ of $\overline G$ connected by an edge $e$. Then for $0\leq i\leq b_{v,v'}$ we have isomorphisms
$$\varphi_i^{e,v}\colon \mathscr L^v(-D_i^{e,v})/\mathscr L^v(-D_{i+1}^{e,v})\rightarrow \mathscr L^{v'}(-D_{b_{v,v'}-i}^{e,v'})/\mathscr L^{v'}(-D_{b_{v,v'}+1-i}^{e,v'}).$$
\end{prop}

Note that in the above proposition, if $ D_{i+1}^{e,v}-D_i^{e,v}=\sum_{e'\in I}P_{e'}^v$ then $ D_{b_{v,v'}-i+1}^{e,v'}-D_{b_{v,v'}-i}^{e,v'}=\sum_{e'\in I}P_{e'}^{v'}$. In particular we have $\deg(D_{i+1}^{e,v}-D_i^{e,v})=\deg(D_{b_{v,v'}-i+1}^{e,v'}-D_{b_{v,v'}-i}^{e,v'})$ for all $i$.

\begin{defn}
(1) Let $Z$ be  a smooth curve and $r,d\geq 0$. Let $D_0\leq\cdots \leq D_{b+1}$ be a sequence of effective divisors on $X$. We say $j$ is \textbf{critical} for $D_\bullet$ if $D_{j+1}\neq D_j$. 

(2) Suppose further that $D_0=0$ and $\deg D_{b+1}>d$. Given $(\mathscr L,V)$ a $\mathfrak g^r_d$ on $Z$, we define the \textbf{multivanishing sequence} of $(\mathscr L,V)$ along $D_\bullet$ to be the sequence 
$$a_0\leq\cdots\leq a_r$$
where a value $a$ appears in the sequence $m$ times if for some $i$ we have $\deg D_i=a$ and $\deg D_{i+1}>a$, and $\dim (V(-D_i)/V(-D_{i+1}))=m$.

(3) Given $s\in V$ nonzero, the \textbf{order of vanishing} $\mathrm{ord}_{D_\bullet}(s)$ along $D_\bullet$ is $\deg D_i$ where $i$ is maximal so that $s\in V(-D_i)$.
\end{defn}

%\begin{rem}\label{symmetric critical}
%Note that in Proposition \ref{glueing isomorphism}, if $ D_{i+1}^{e,v}-D_i^{e,v}=\sum_{e'\in I}P_{e'}^v$ then $ D_{b_{v,v'}-i+1}^{e,v'}-D_{b_{v,v'}-i}^{e,v'}=\sum_{e'\in I}P_{e'}^{v'}$. In particular we have $$\deg(D_{i+1}^{e,v}-D_i^{e,v})=\deg(D_{b_{v,v'}-i+1}^{e,v}-D_{b_{v,v'}-i}^{e,v})$$ for all $i$. Hence $j$ is critical for $D_\bullet^{e,v}$ if and only if $b_{v,v'}-j$ is critical for $D_\bullet^{e,v'}$.
%\end{rem}
One checks easily that $j$ is critical for $D_\bullet^{e,v}$ if and only if $b_{v,v'}-j$ is critical for $D_\bullet^{e,v'}$.
We are now able to state the definition of limit linear series on a curve of pseudocompact type (cf. \cite[Definition 2.16]{osserman2014dimension}):

\begin{defn}\label{limit linear series on curves}
Suppose we have a tuple $(\mathscr L, (V_v)_{v\in V(G)})$ with $\mathscr L$ a line bundle of multidegree $w_0$ on $\widetilde X_0$, and each $V_v$ is a $(r+1)$-dimensional space of global sections of $\mathscr L^v$ as in Proposition \ref{glueing isomorphism}. For each pair $(e,v)$ in $\overline G$ where $v$ is a vertex of $e$, let $a_0^{e,v},...,a_r^{e,v}$ be the multivanishing sequence of $V_v$ along $D^{e,v}_\bullet$. Then $(\mathscr L, (V_v)_{v\in V(G)})$ is a \textbf{limit linear series} of multidegree $w_0$ with respect to $(w_v)_{v\in V(G)}$ on $(X_0,\bm n)$ if for any $e\in E(\overline G)$ with vertices $v$ and $v'$ we have:

\ (I) for $l=0,...,r,$ if $a_l^{e,v}=\deg D_j^{e,v}$ with $j$ critical for $D_\bullet^{e,v}$, then $a_{r-l}^{e,v'}\geq \deg D_{b_{v,v'}-j} ^{e,v'}$;

(II) there exists bases $s_0^{e,v},...,s_r^{e,v}$ of $V_v$ and $s_0^{e,v'},...,s_r^{e,v'}$ of $V_{v'}$ such that 
$$\mathrm{ord}_{D_\bullet^{e,v}}s_l^{e,v}=a_l^{e,v}\ \ \ \ \mathrm{for}\ \ l=0,...,r,$$
and similarly for $s_l^{e,v'}$, and for all $l$ such that $a_{r-l}^{e,v'}= \deg D_{b_{v,v'}-j} ^{e,v'}$ (where $j$ as in (I)) we have $\varphi_j^{e,v}(s_l^{e,v})=s_{r-l}^{e,v'}$ where we consider $s_l^{e,v}\in V_v(-D_j^{e,v})$ and $s_{r-l}^{e,v'}\in V_{v'}(-D_{b_{v,v'}-j}^{e,v'})$ and $ \varphi_j^{e,v}$ is as in Proposition \ref{glueing isomorphism}.
\end{defn}

\begin{rem}\label{def of lls}
Let $g_j=\#\{0\leq l\leq r|a_l^{e,v}=\deg D_j^{e,v}\ \mathrm{and}\ a_{r-l}^{e,v'}= \deg D_{b_{v,v'}-j} ^{e,v'}\}$. Using the identification of the two linear spaces induced by $\varphi_j^{e,v}$, the condition (2) in the above definition is equivalent to that the space $$(V_v(-D_j^{e,v})/V_v(-D_{j+1}^{e,v}))\cap (V_{v'}(-D_{b_{v,v'}-j}^{e,v'})/V_v(-D_{b_{v,v'}-j+1}^{e,v'}))$$has at least dimension $g_j$.
\end{rem}

\subsection{Limit linear series on metrized complexes.} In this subsection we introduce Amini and Baker's construction in \cite{amini2015linear} of metrized complexes and limit linear series on them. Suppose for now that $\kappa$ is algebraically closed.
Recall that a \textbf{metrized complex} $\mathfrak C$ of curves over $\kappa$ consists of the following data: (1) a metric graph $\Gamma$ with underlying graph $G$; (2) for each vertex $v$ of $G$ a smooth curve $C_v$ over $\kappa$; (3) for each vertex $v$ of $G$, a bijection $ e\mapsto x_{e}^v$ between the edges of $G$ incident on $v$ (recall that $G$ is assumed loopless in the beginning of this section) and a subset  $A_v$ of $C_v(\kappa)$. For consistency of symbols we assume that $\Gamma$ is of integral edge lengths.

The geometric realization $|\mathfrak C|$ of $\mathfrak C$ is the union of the edges of $\Gamma$ and the collection of curves $C_v$, with each endpoint $v$ of $e$ identified with $x_e^v$. The following is a geometric realization of a metrized complex which has rational $C_v$ for all $v\in V(G)$ and whose underlying graph $G$ is $K_2$.
$$
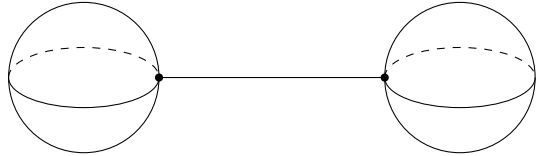
\begin{tikzpicture}
\draw (0,0) circle (1cm);
\draw (-1,0) arc (180:360:1cm and 0.4cm);
\draw[dashed] (-1,0) arc (180:0:1cm and 0.4cm);
\draw (5,0) circle (1cm);
\draw (4,0) arc (180:360:1cm and 0.4cm);
\draw[dashed] (4,0) arc (180:0:1cm and 0.4cm);
\draw (0: 1) node[circle, fill=black, scale=0.3,]{};
\draw (0: 1) -- (0: 4);
\draw (0: 4) node[circle, fill=black, scale=0.3,]{};
\end{tikzpicture}
$$ 

\begin{ex}\label{curve complex}
Given a curve $X_0$ with dual graph $G$ and a chain structure $\bm n$ on $G$, we can associate a metrized complex $\mathfrak C_{X_0,\bm n}$: let $\Gamma$ be the metric graph with underlying graph $G$ and edge lengths given by $\bm n$, let $C_v$ be the component $Z_v$ and $x_e^v=P_e^v$ (hence $A_v=\mathcal A_v$). If the edge lengths are all $1$ we denote by $\mathfrak C_{X_0}$ instead.
\end{ex}
A divisor $\mathcal D$ on $\mathfrak C$ is a finite formal sum of points in $|\mathfrak C|$. Denoting $\mathcal D=\sum_{x\in |\mathfrak C|}a_x(x)$, we can naturally associate a divisor $\mathcal D_\Gamma$ on $\Gamma$, called the $\Gamma$-\textbf{part} of $\mathcal D$, as well as, for each $v\in V(G)$, a divisor $\mathcal D_v$ on $C_v$ called the $C_v$-\textbf{part} of $\mathcal D$ as follows:
$$\mathcal D_v=\sum_{x\in C_v(\kappa)}a_x(x)\ \ \mathrm{and}\ \ \mathcal D_\Gamma=\sum_{x\in \Gamma\backslash V(G)}a_x(x)+\sum_{v\in V(G)}\deg (\mathcal D_v)(v).$$  
Note that we have $\deg \mathcal D=\deg \mathcal D_\Gamma$.

A \textbf{nonzero rational function} $\mathfrak f$ on $\mathfrak C$ is the data of a rational function $\mathfrak f_\Gamma$ (the $\Gamma$-\textbf{part}) on $\Gamma$ and nonzero rational functions $\mathfrak f_v$ (the $C_v$-\textbf{part}) on $C_v$ for each $v\in V(G)$. The divisor associated to $\mathfrak f$ is defined to be 
$$\mathrm{div}(\mathfrak f)\colon =\sum_{x\in|\mathfrak C|}\mathrm{ord}_x(\mathfrak f)(x)$$
where $\mathrm{ord}_x(\mathfrak f)$ is as follows: if $x\in \Gamma\backslash V(G)$ then $\mathrm{ord}_x(\mathfrak f)=\mathrm{ord}_x(f_\Gamma)$; 
% i.e. the sum of the slops of $f_\Gamma$ in all tangent directions emanating from $x$;
 if $x\in C_v(\kappa)\backslash A_v$ then $\mathrm{ord}_x(\mathfrak f)=\mathrm{ord}_x(\mathfrak f_v)$; if $x=P_e^v\in \mathcal A_v$ then $\mathrm{ord}_x(\mathfrak f)=\mathrm{ord}_x(\mathfrak f_v)+\mathrm{slp}_{e,v}(\mathfrak f_\Gamma)$. In particular the $\Gamma$-part of $\mathrm{div}(\mathfrak f)$ is equal to $\mathrm{div}(\mathfrak f_\Gamma)$. 

Similarly to the divisors on graphs we have the following definitions: 
 
\begin{defn}
Suppose $\Gamma$ has integral edge lengths.
A divisor $\mathcal D$ on $\mathfrak C$ is \textbf{rational, integral, and edge-reduced} if $\mathcal D_\Gamma$ is. Similarly, let $\bm n$ be the chain structure on $G$ induced by $\Gamma$ and $w$ an admissible multidegree on $(G,\bm n)$, then $\mathcal D$ is \textbf{of multidegree} $w$ if $\mathcal D_\Gamma=D_w$. 
\end{defn}
 
\begin{defn}
Divisors of the form $\mathrm{div}(\mathfrak f)$ are called \textbf{principal}. Two divisors in $\mathrm{Div}(\mathfrak C)$ are called \textbf{linearly equivalent} if they differ by a principal divisor. The \textbf{rank} $r_{\mathfrak C}(\mathcal D)$ of a divisor $\mathcal D$ is the largest integer $r$ such that $\mathcal D-\mathcal E$ is linearly equivalent to an effective divisor for all effective divisor $\mathcal E$ of degree $r$ on $\mathfrak C$.
\end{defn}

The definition of the rank of $\mathcal D$ can be refined by restricting the set of rational functions on $\mathfrak C$ that induce linear equivalence, as follows:

\begin{defn}
Suppose we are given, for each $v\in V(G)$, a non-empty $\kappa$-linear subspace $F_v$ of $K(C_v)$. Denote by $\mathcal F$ the collection of all $F_v$. We define the $\mathcal F$-\textbf{rank} $r_{\mathfrak C,\mathcal F}(\mathcal D)$ of $\mathcal D$ to be the maximum integer $r$ such that for every effective divisor $\mathcal E$ of degree $r$, there is a nonzero rational function $\mathfrak f$ on $\mathfrak C$ with $\mathfrak f_v\in F_v$ for all $v\in V(G)$ such that $\mathcal D+\mathrm{div}(\mathfrak f)-\mathcal E\geq 0$.
\end{defn}

Note that by definition we always have $r_{\mathfrak C,\mathcal F}(\mathcal D)\leq r_{\mathfrak C}(\mathcal D)\leq r(\mathcal D_\Gamma)$. We are now able to define limit linear series on metrized complexes:

\begin{defn}
A \textbf{limit linear series} of degree $d$ and rank $r$ on $\mathfrak C$ is a (equivalence class of) pair $(\mathcal D,\mathcal H)$ consisting of a divisor $\mathcal D$ of degree $d$ and a collection $\mathcal H$ of $(r+1)$-dimensional subspaces $H_v\subset K(C_v)$ for all $v\in V(G)$, such that $r_{\mathfrak C,\mathcal H}(\mathcal D)=r$. Two pairs $(\mathcal D,\mathcal H)$ and $(\mathcal D',\mathcal H')$ are considered equivalent if there is a rational function $\mathfrak f$ on $\mathfrak C$ such that $\mathcal D'=\mathcal D+\mathrm{div}(\mathfrak f)$ and $H_v=H_v'\cdot \mathfrak f_v$ for all $v\in V(G)$.
\end{defn}

\begin{defn}
A limit linear series on $\mathfrak C$ is \textbf{of multidegree} $w_0$ if there is a representative $(\mathcal D,\mathcal H)$ such that $\mathcal D$ is so.
\end{defn}

Let $X$ be a smooth curve over $\widetilde K$. Recall that a \textbf{strongly semistable model} for $X$ is a flat and integral proper relative curve over $\widetilde R$ whose generic fiber is isomorphic to $X$ and whose special fiber is a curve in our setting (or a strongly semistable curve over $\kappa$ as in \cite[\S 4.1]{amini2015linear}). Given a strongly semistable model $\mathfrak X$, again there is a associated metrized complex $\mathfrak C\mathfrak X$ where the underlying graph $G$ is the dual graph of the special fiber $X_0=\overline {\mathfrak X}$ and $C_v$ is the component $Z_v$ of $X_0$ and $x_e^v=P_e^v$. Moreover, the length $l(e)$ of $e\in E(\Gamma)=E(G)$ is val$(\omega_e)$ for some $w_e\in \widetilde R$ such that the local equation of $\mathfrak X$ at the node $P_e$ of $e$ is $xy-\omega_e$. Equivalently, consider the natural reduction map $\mathrm{red}\colon X(\widetilde K)\rightarrow X_0(\kappa)$ induced by the bijection between $X(\widetilde K)$ and $\mathfrak X(\widetilde R)$; this extends to a map $\mathrm{red}\colon X^{\mathrm{an}}\rightarrow X_0$ and $l(e)$ is the modulus of the open (analytic) annulus $\mathrm{red}^{-1}(P_e)$ in $X^{\mathrm{an}}$, where $X^{\mathrm{an}}$ is the Berkovich analytification of $X$. Note that any metrized complex with edge lengths contained in $\mathrm{val}(\widetilde K)$ can be constructed from strongly semistable models (cf. \cite[Theorem 4.1]{amini2015linear}). 

There is a canonical embedding of $\Gamma$ into $X^{\mathrm{an}}$ as well as a canonical retraction map $\tau\colon X^{\mathrm{an}}\rightarrow \Gamma$, which induces by linearity a specialization map $\tau_*\colon \mathrm{\mathrm{Div}}(X)\rightarrow \mathrm{Div}(\Gamma)$ which maps $X(\widetilde K)$ to the set of rational points of $\Gamma$. For $P\in X(\widetilde K)$, if $\tau_*(P)=v\in V(G)$ then $\mathrm{red}(P)$ is a nonsingular closed point of $X_0$ in $C_v$. We thus have the \textbf{specialization map} of divisors 
$\tau_*^{\mathfrak C\mathfrak X}\colon\mathrm{Div}( X)\rightarrow \mathrm{Div}(\mathfrak C\mathfrak X)$ given by the linearly extension of:
$$\tau_*^{\mathfrak C\mathfrak X}(P)=\left\{\begin{array}{ll}\tau_*(P) \ \ \ \ \tau_*(P)\not\in V(G)\\\mathrm{red}(P)\ \ \tau_*(P)\in V(G).\end{array}\right.$$

On the other hand, let $x\in X^{\mathrm{an}}$ be a point of type $2$. The completed residue field $\widetilde{\mathcal H(x)}$ of $x$ has transcendence degree one over $\kappa$ and corresponds to a curve $C_x$ over $\kappa$. Given a nonzero rational function $f$ on $X$, choose $c\in \widetilde K^\times$ such that $|f(x)|=|c|$. We denote $f_x$ as the image of $c^{-1}f$ in $K(C_x)\cong\widetilde{\mathcal H(x)}$, which is well defined up to scaling by $\kappa^\times$. We call $f_x$ the \textbf{normalized reduction} of $f$. Note that the normalized reduction of a $\widetilde K$-vector space $V$ is a $\kappa$-vector space of the same dimension, whereas the normalized reduction of a basis of $V$ is not necessarily a basis of the normalized reduction of $V$. See \cite[Lemma 4.3]{amini2015linear} for details.

\begin{defn} Given $X$ and $\mathfrak X$ as above.
The \textbf{specialization} $\tau_*^{\mathfrak C\mathfrak X}(f)$ of $0\neq f \in K(X)$ is a nonzero rational function on $\mathfrak C\mathfrak X$ whose $\Gamma$-part is the restriction to $\Gamma$ of the piecewise linear function $F=\mathrm{log}|f|$ on $X^{\mathrm{an}}$ and whose $Z_v$-part is (up to multiplication by $\kappa ^\times$) the normalized reduction $f_{x_v}$ in which $x_v$, whose completed residue field is identified with the function field of $Z_v$, is the image of $v$ under the embedding $\Gamma\hookrightarrow X^{\mathrm{an}}$ mentioned above. 
The \textbf{specialization} $\tau_*^{\mathfrak C\mathfrak X}(V)$ of a linear subspace $V$ of $K(X)$ is a collection $\{F_v\}_{v\in V(G)}$ of space of rational functions on $Z_v$, where $F_v=\{f_{x_v}|f\in V\}$. We denote by $\tau_{*v}^{\mathfrak C\mathfrak X}$ the $Z_v$-part of $\tau_*^{\mathfrak C\mathfrak X}$, or equivalently the normalized reduction to $\widetilde{\mathcal H(x_v)}$.

%$\tau_*^{\mathfrak C\mathfrak X}(V)$ of a linear subspace $V$ of $K(X)$ is a collection $\{F_v\}_{v\in V(G)}$ of space of rational functions on $Z_v$, where $F_v=\{f_{x_v}|f\in V\}$ in which $x_v$, whose completed residue field is identified with the function field of $Z_v$, is the image of $v$ under the embedding $\Gamma\hookrightarrow X^{\mathrm{an}}$ mentioned above.   
\end{defn}    

Note that $\tau_{*v}^{\mathfrak C\mathfrak X}(V)$ is a linear space of the same dimension as $V$ as mentioned above. Similar to the classical case, we have that the specialization of a linear series is a limit linear series:
 
\begin{thm}
$\mathrm{(}$\cite[Theorem 5.9]{amini2015linear}$\mathrm{)}$ Let $X,\mathfrak X$ be as above. Let $D$ be a divisor on $X$ and $(\mathscr O_X(D),V)$ be a $\mathfrak g^r_d$ on $X$, where $V\subset H^0(X,\mathscr O_X(D))\subset K(X).$ Then the pair $(\tau_*^{\mathfrak C\mathfrak X}(D),\tau_*^{\mathfrak C\mathfrak X}(V))$ is a limit $\mathfrak g^r_d$ on $\mathfrak C\mathfrak X$.
\end{thm}
We end with the definition of smoothability of limit linear series on metrized complexes: 
\begin{defn}
A limit linear series $(\mathcal D,\mathcal H)$ of degree $d$ and rank $r$ on a metrized complex $\mathfrak C$ over $\kappa$ is called \textbf{smoothable} if there exists a smooth proper curve $X$ over $\widetilde K$ such that $\mathfrak C=\mathfrak C\mathfrak X$ for some strongly semistable model $\mathfrak X$ of $X$ and $(\mathcal D,\mathcal H)$ arises the specialization of a limit $\mathfrak g^r_d$ on $X$.
\end{defn}

\section{The weak glueing condition} 
In this section we introduce the concept of pre-limit linear series on a curve $X_0$ of pseudocompact type with given chain structure $\bm n$ (on the dual graph), and a forgetful map from the space of pre-limit linear series on $(X_0,\bm n)$  to the set of limit linear series on $\mathfrak C_{X_0,\bm n}$, derived from \cite{osserman2017limit}.
%given in \textcolor{red}{ref}, which is from the space of Osserman limit linear series to the set of Amini-Baker limit linear series on the corresponding metrized complex, to the space of pre-limit linear series. 
We show that this map is bijective onto the set of $(\mathcal D,\mathcal H)$ with integral $\mathcal D$. By this we can extend the weak glueing condition on curves, which will be defined in the beginning of this section, to metrized complexes, and transfer the smoothing problems on metrized complexes to those on curves. To be consistent with Section 2.3 we assume in this section that $\kappa$ is algebraically closed. We use the following notation:
\begin{nota}\label{notation}
Let $X_0$ be a curve over $\kappa$ of pseudocompact type, $G$ the dual graph of $X_0$ and $\overline G$ the tree obtained from $G$ as in \S 1.1.  
Let $\bm n$ be a chain structure on $G$ and $\Gamma$ the induced metric graph. Take $w_0$ an admissible multidegree of total degree $d$ on $(G,\bm n)$. %and $(\mathscr O_v)_{v\in V(G)}$ an enriched structure on $\widetilde X_0$ 
  \ Choose also a tight tuple $(w_v)_{v\in V(G)}$ of admissible multidegrees in $V(G(w_0))$.%obtained by twisting from $w_0$ such that with each $w_v$ is concentrated at $v$. Moreover, we assume that $(w_v)_v$ is as in Proposition \ref{glueing isomorphism}.
 \ Let $d_v$ be the coefficient of $D_{w_v}$ at $v$. Let $\widetilde X_0$ be the nodal curve corresponding to $\bm n$ as constructed in Section 2.2 (after Definition \ref{enriched structure}) and $\widetilde G$ the dual graph of $\widetilde X_0$. Let $\mathfrak C_{X_0,\bm n}$ and $\mathfrak C_{\widetilde X_0}$ be the induced metrized complexes as in Example \ref{curve complex}. Given an enriched structure on $\widetilde X_0$, for a line bundle $\mathscr L$ on $\widetilde X_0$ recall that $\mathscr L^v=\mathscr L_{w_v}|_{Z_v}$ is constructed as in Proposition \ref{glueing isomorphism}.
\end{nota}
 
With the notation above, we start from defining pre-limit linear series on $(X_0,\bm n)$ and the weak glueing condition.

\begin{defn}\label{pre-limit linear series} 
A \textbf{pre-limit linear series} of rank $r$ with respect to the admissible multidegrees $(w_v)_v$ on $(X_0,\bm n)$ is a tuple $(\mathscr L_v, V_v)_{v\in V(G)}$ of line bundles $\mathscr L_v$ on $Z_v$ of degree $d_v$ and linear spaces $V_v\subset H^0(Z_v,\mathscr L_v)$ of dimension $r+1$ such that the multivanishing sequence $a_\bullet^{e,v}$ of $V_v$ along the twisting divisors $D_\bullet ^{e,v}$ satisfies condition (I) of Definition \ref{limit linear series on curves}. Given an enriched structure, we say a pre-limit linear series \textbf{lifts to a limit linear series} if there exists a line bundle $\mathscr L$ on $\widetilde X_0$ of multidegree $w_0$ such that $\mathscr L_v=\mathscr L^{v}$, and the tuple $(\mathscr L, (V_v)_v)$ is a limit linear series as in Definition \ref{limit linear series on curves}. 
\end{defn}
 
\begin{defn}\label{weak glueing condition}
%Let the admissible multidegrees $(w_v)_v$ be as in \cite[Situation 2.6]{osserman2014dimension}. 
%For $j$ critical for $D_\bullet^{e,v}$ let $g_j=\#\{l\in\{0,...,r\}|a_l^{e,v}=\deg D_j^{e,v}\ \mathrm{and}\ a_{r-l}^{e,v}=\deg D_{b_{v,v'}-j}^{e,v}\}$. 
We say that a pre-limit linear series $(\mathscr L_v, V_v)_{v\in V(G)}$ satisfies the \textbf{weak glueing condition} if for all $e\in \overline G$, let $v$ and $v'$ be the vertices adjacent to $e$, then for all critical $j$ with respect to $D_\bullet ^{e,v}$, there exists subspaces $$W_v\subset V_v(-D_j^{e,v})/V_v(-D_{j+1}^{e,v})\mathrm{\ and\ }W_{v'}\subset V_{v'}(-D_{b_{v,v'}-j}^{e,v'})/V_{v'}(-D_{b_{v,v'}-j+1}^{e,v'})$$ both of which have dimension $g_j$, where $g_j$ is defined in Remark \ref{def of lls}, such that for all torus orbits $T_v\subset \mathscr L_v(-D_j^{e,v})/\mathscr L_v(-D_{j+1}^{e,v})$ and $T_{v'}=\varphi_j^{e,v}(T_v)\subset \mathscr L_{v'}(-D_{b_{v,v'}-j}^{e,v'})/\mathscr L_{v'}(-D_{b_{v,v'}-j+1}^{e,v'})$, where $\varphi_j^{e,v}$ as in Proposition \ref{glueing isomorphism}, we have $\dim(T_v\cap W_v)=\dim(T_{v'}\cap W_{v'}).$
\end{defn}
 
In the above definition, suppose $D_{j+1}^{e,v}- D_{j}^{e,v}=P_{e_1}^v+\cdots+P_{e_m}^v$ where $e_i\in E(G)$ lies over $e$. We have an isomorphism $$\psi\colon H^0(Z_v,\mathscr L_v(-D_j^{e,v})/\mathscr L_v(-D_{j+1}^{e,v}))\rightarrow \kappa^m$$ given by $\psi(s)=(s(P_{e_1}^v),...,s(P_{e_m}^v))$ which is unique up to coordinate-wise scaling. This induces an action of the $m$-dimensional torus on $\mathscr L_v(-D_j^{e,v})/\mathscr L_v(-D_{j+1}^{e,v})$ which is independent of $\psi$. Similarly we have a torus action on $\mathscr L_{v'}(-D_{b_{v,v'}-j}^{e,v'})/\mathscr L_{v'}(-D_{b_{v,v'}-j+1}^{e,v'})$ and $\varphi_j^{e,v}$ respects torus orbits.

\begin{rem}\label{weak glueing remark}
In the above definition it is enough to check that for all torus orbits $T_v$ we have $T_v\cap W_v$ is nonempty if and only if $T_{v'}\cap W_{v'}$ is nonempty: let $T_v=\{(x_1,...,x_m)|x_i\neq 0\ \mathrm{if\ and \ only\ if\ }i\in I\}$ for some $I\subset \{1,...,m\}$ and $\dim(T_v\cap W_v)=l$. For $s=(a_1,...,a_m)\in \mathscr L_v(-D_j^{e,v})/\mathscr L_v(-D_{j+1}^{e,v})$ denote $I_s=\{i|a_i\neq 0\}$. We can find $s_1,...,s_l\in T_v\cap W_v$ that are linearly independent. Taking linear combinations of $\{s_i\}$ gives $t_1,...,t_l\in W_v$ such that $I_{t_1}\subsetneq I_{t_2}\subsetneq\cdots\subsetneq I_{t_l}= I$. Now let $T_{v'}^i= \{(x_1',...,x_m')|x_i'\neq 0\ \mathrm{if\ and \ only\ if\ }i\in I_{t_i}\}\subset \overline T_{v'}$. We have that $T_{v'}^l=T_{v'}$ and $W_{v'}\cap T_{v'}^i\neq \emptyset$ by assumption, which implies that $\dim(T_{v'}\cap W_{v'})\geq l=\dim(T_{v}\cap W_{v})$ since $W_{v'}$ is a $\kappa$-vector space. Same argument shows that $\dim(T_{v}\cap W_{v})\geq\dim(T_{v'}\cap W_{v'})$, hence $\dim(T_{v}\cap W_{v})=\dim(T_{v'}\cap W_{v'})$.
\end{rem}  

\begin{ex}
Note that $g_j\leq \deg D_{j+1}^{e,v}- \deg D_{j}^{e,v}$. If $g_j= \deg D_{j+1}^{e,v}- \deg D_{j}^{e,v}$ then the weak glueing condition is trivial, since in this case we have $$W_v=\mathscr L_v(-D_j^{e,v})/\mathscr L_v(-D_{j+1}^{e,v})\mathrm{\ and\ } W_{v'}=\mathscr L_{v'}(-D_{b_{v,v'}-j}^{e,v'})/\mathscr L_{v'}(-D_{b_{v,v'}-j+1}^{e,v'}).$$
On the other hand, if $g_j=1$ then 
the weak glueing condition is equivalent to the existence of $s_v\in
 V_v(-D_j^{e,v})/V_v(-D_{j+1}^{e,v})$ and $s_{v'}\in V_{v'}(-D_{b_{v,v'}-j}^{e,v'})/V_{v'}(-D_{b_{v,v'}-j+1}^{e,v'})$ such that 
$s_v$ vanishes at $P_v^{\tilde e}$ if and only if $s_{v'}$ vanishes at $P_{v'}^{\tilde e}$, where $P_v^{\tilde e}$ runs over the support of $D_{j+1}^{e,v}-D_j^{e,v}.$
\end{ex}

See also Example \ref{example} for a concrete case of a pre-limit linear series (as well as a limit linear series on $\mathfrak C_{X_0,\bm n}$) that does not satisfy the weak glueing condition.

\begin{prop}\label{lls weak}
A pre-limit linear series that lifts to a limit linear series must satisfy the weak glueing condition.
\end{prop}
\begin{proof}
Follows directly from Remark \ref{def of lls}.
\end{proof}
%$l\in\{0,...,r\}$ such that $a_l^{e,v}=\deg D_j^{(e,v)}$ for some $j$ critical and $a_{r-l}^{(e,v')}=\deg D_{b_{v,v'}-j}^{(e,v')}$, there is a $s_l^{(e,v)}\in V_v$ and $s_{r-l}^{(e,v')}\in V_{v'}$ such that: 

%(1) we have $\mathrm{ord}_{D_\bullet^{(e,v)}} s_l^{(e,v')}=a_l^{(e,v)}$ and $\mathrm{ord}_{D_\bullet^{(e,v')}} s_{r-l}^{(e,v')}=a_{r-l}^{(e,v')}$; 

%(2) Let $P_i\in Z_v$ and $Q_i\in Z_{v'}$ be the marked points corresponds to $e_i\in E(\Gamma)$ over $e$, then $\mathrm{ord}_{P_i}({s_l^{(e,v)}})-\mathrm{ord}_{P_i}({D_j^{(e,v)}})>0$ if and only if $\mathrm{ord}_{P_i}({s_{r-l}^{(e,v')}})-\mathrm{ord}_{Q_i}({D_{b_{v,v'}-j}^{(e,v')}})>0$ for all $i$ such that $\mathrm{ord}_{P_i}({D_{j+1}^{(e,v)}})-\mathrm{ord}_{P_i}({D_j^{(e,v)}})=\mathrm{ord}_{Q_i}({D_{b_{v,v'}-j+1}^{(e,v')}})-\mathrm{ord}_{Q_i}({D_{b_{v,v'}-j}^{(e,v')}})>0.$

Let $P_d^r(X_0,\bm n,(w_v)_v)$ be the space\footnote{Here is a description of $P^r_d(X_0,\bm n,(w_v)_v)$: take non-negative sequences $a_\bullet ^{e,v}$ for each adjacent pair $(e,v)$ in $\overline G$ satisfying condition (I) of Definition \ref{limit linear series on curves}, let $G^r_d(Z_v,(D_\bullet^{e,v},a_\bullet^{e,v})_{e})$ be the space of $\mathfrak g^r_{d_v}$s on $Z_v$ having multivanishing sequence at least $a_\bullet^{e,v}$ along $D_\bullet^{e,v}$ for all $e$ incident on $v$, then $P^r_d(X_0,\bm n,(w_v)_v)$ is the union over all such $(a_\bullet^{e,v})_{e,v}$ of $$\prod_{v\in V(G)} G^r_{d_v}(Z_v,(D_\bullet^{e,v},a_\bullet^{e,v})_{e})$$ in $\prod_{v\in V(G)}G^r_{d_v}(Z_v)$ where $G^r_{d_v}(Z_v)$ is the space of limit $\mathfrak g^r_{d_v}$s on $Z_v$.}
 of pre-limit $\mathfrak g^r_d$s with respect to $(w_v)_v$ on $(X_0,\bm n)$, and $\mathfrak G^r_d(\mathfrak C_{X_0,\bm n})$ be the set of limit $\mathfrak g^r_d$s on $\mathfrak C_{X_0,\bm n}$. Given $w,w'\in V(G(w_0))$, we denote $D_{w,w'}^v$ the divisor on $Z_v$ obtained by twisting from $w$ to $w'$. More precisely, if $D_{w'}=D_w+\mathrm{div}(f)$ where $f$ is a piecewise linear function on $\Gamma$ then $D_{w,w'}^v=\sum\mathrm{slp}_{e,v}(f)P_e^v$ where the summation is taking over all edges in $G$ that is incident on $v$. In other words, if $P(w,v_1,...,v_m)=(w_1,\mu_1),...,(w_{m+1},\mu_{m+1})$ is a path in $G(w_0)$ from $w$ to $w'$, let $S\subset \{1,...,m\}$ consists of $i$ such that $v_i$ $(\neq v)$ is adjacent to $v$ and for $i\in S$ let $e_i\in E(\overline G)$ be the edge connecting $v$ and $v_i$, then (\cite[Notation 4.7]{osserman2017limit})
$$D_{w,w'}^v=\sum_{i\in S}\ \sum_{\begin{tiny}\begin{array}{cc}\tilde e\in E(G)\mathrm{\ over\ } e_i\\ \mu_{i+1}(\tilde e)=0\end{array}\end{tiny}}P_{\tilde e}^{v}-\sum_{i\colon v_i=v}\ \sum_{\begin{tiny}\begin{array}{cc}\tilde e\in E(G) \mathrm{\ incident\ on\ }v\\ \mu_i(\tilde e)=0\end{array}\end{tiny} }P_{\tilde e}^v.$$

\begin{defn}\label{forgetful map}
We define a map $\mathfrak F_{(w_v)_v}\colon P^r_d(X_0,\bm n,(w_v)_v)\rightarrow \mathfrak G^r_d(\mathfrak C_{X_0,\bm n})$ as follows: Given a tuple $(\mathscr L_v, V_v)_v\in P^r_d(X_0,\bm n,(w_v)_v) $, for each $v\in V(G)$, choose nonzero $s_v\in V_v$. Fix $w\in V(G(w_0))$. Take $\mathcal D\in \mathrm{Div}(\mathfrak C_{X_0,\bm n})$ of multidegree $w$ such that $\mathcal D_v=\mathrm{div}^0(s_v)-D^v_{w,w_v}$ for all $v$ and a collection $\mathcal H=(H_v)_v$ of rational functions with $H_v=\{\frac{s}{s_v}\colon s\in V_v\}$.
%let $D_v'\in \mathrm{Div}(Z_v)$ such that $\mathcal O(D_v')\cong \mathscr L_v$. Let $H_v\subset K(Z_v)$ be the space of rational functions that corresponds to $V_v\subset H^0(Z_v,\mathcal O(D_v'))$. %Let $D_v=D_v'-D_{w_v}$ where $D_{w_v}$ is the divisor on $Z_v$ obtained by twisting from $w_0$ to $w_v$. 
%Let $\mathcal D$ be the divisor on $\mathfrak C_{X_0,\bm n}$ such that $\mathcal D_v=D_v'-D_{v}$ where $D_{v}$ is the divisor on $Z_v$ obtained by twisting from $w_0$ to $w_v$ and $\mathcal D_\Gamma=w_0$. More precisely, if $w_v=w_0+\mathrm{div}(f)$ where $f$ is a piecewise linear function on $\Gamma$ then $D_v=\sum\mathrm{slp}_{e,v}(f)P_e^v$ where the summation is taking over all edges in $G$ that is adjacent to $v$.  Let also $\mathcal H=(H_v)_v$.
 Then we define $\mathfrak F_{(w_v)_v}((\mathscr L_v, V_v))=(\mathcal D,\mathcal H)$. We sometimes write $\mathfrak F$ instead of $\mathfrak F_{(w_v)_v}$ when $(w_v)_v$ is specified.
\end{defn}

Note that by \cite[Propostion 4.12]{osserman2017limit} and \cite[Propostion 5.9]{osserman2017limit} the map $\mathfrak F_{(w_v)_v}$ is well defined, i.e. $\mathfrak F_{(w_v)_v}((\mathscr L_v, V_v)_v)$ is a limit $\mathfrak g^r_d$ on $\mathfrak C_{X_0,\bm n}$ for all $(\mathscr L_v, V_v)_v\in P^r_d(X_0,\bm n,(w_v)_v)$, and it is independent of choices of $s_v$ and $w$. We next show that $\mathfrak F_{(w_v)_v}$ is a bijection onto the set of limit $\mathfrak g^r_d$s on $\mathfrak C_{X_0,\bm n}$ of multidegree $w_0$.
 
%\begin{defn}
%Let $\mathcal D$ (resp. $D$) be a divisor on $\mathfrak C_{X_0,\bm n}$ (resp. $\Gamma$), we say that $\mathcal D$ (resp. $D$) is \textbf{edge-reduced} if the restriction of $\mathcal D$ (resp. $D$) on $\mathfrak C_{X_0,\bm n}\backslash \cup_v Z_v$ (resp. $\Gamma\backslash V(G)$) is either empty or an effective divisor of degree one. We say that $\mathcal D$ is \textbf{rational} if $\mathcal D_\Gamma$ is rational, and \textbf{integral} if $\mathcal D_\Gamma$ is supported on integer point.
%\end{defn}

%\begin{prop}
%There is a one-to-one corresponds between the space of integral edge-reduced divisors on a metrized complex and the space of (admissible) multidegrees on the induced metric graph that preserves the total degree.
%\end{prop}

%\begin{defn}
%We say that a integral edge-reduced divisor on a metrized complex is \textbf{of multidegree} $w_0$ if the multidegree induced by $\mathcal D_\Gamma$ is $w_0$.
%\end{defn}
Given $(\mathcal D,\mathcal H)$ a limit $\mathfrak g^r_d$ on $\mathfrak C_{X_0,\bm n}$ with $\mathcal D
_\Gamma=D_{{w_0}}$. Recall that $w_v^{\mathrm{red}}\in V(G(w_0))$ is the admissible multidegree such that $D_{w_v^{\mathrm{red}}}$ is $v$-reduced on $\Gamma$. We have that $(w_v^{\mathrm{red}})_v$ is a tight tuple by Proposition \ref{reduced concentrated}. Let $D_v^{\mathrm{red}}=D_{w_0,w_v^{\mathrm{red}}}^v$. We have:
 
\begin{lem}\label{reduced multidegree}
For any $v\in V(G)$ we have $H_v\subset H^0(Z_v,\mathscr O_{Z_v}(\mathcal D_v+D_v^{\mathrm{red}}))$.
\end{lem}
\begin{proof}
Take $v_0\in V(G)$ and a rational function $f$ on $\Gamma$ such that $D_{w_{v_0}^{\mathrm{red}}}=D_{w_0}+\mathrm{div}(f)$. Take rational function $\mathfrak f$ on $\mathfrak C_{X_0,\bm n}$ such that $\mathfrak f_\Gamma=f$ and $\mathfrak f_v=1$ for all $v\in V(G)$. Let $\mathcal D'=\mathcal D+\mathrm{div}(\mathfrak f)$. We then have $\mathcal D'_{v_0}=\mathcal D_{v_0}+D_{v_0}^{\mathrm{red}}$ and $(\mathcal D',\mathcal H)$ is also a limit $\mathfrak g^r_d$, and $\mathcal D'_\Gamma$ is the $v_0$ reduced divisor on $\Gamma$ that is linearly equivalent to $\mathcal D_\Gamma$.

Take $r$ general points $P_1,...,P_r\in Z_{v_0}$, there is a $\mathfrak h\in \mathrm{Rat}(\mathfrak C_{X_0,\bm n})$ such that $\mathcal D'+\mathrm{div}(\mathfrak h)-P_1-\cdots-P_r\geq 0$ and that $\mathfrak h_{v}\in H_{v}$ for all $v$. It follows that $\mathcal D'_\Gamma +\mathrm{div}(\mathfrak h_\Gamma)-rv_0$ is an effective divisor on $\Gamma$. We denote this divisor by $D$ and hence $\mathcal D'_\Gamma=D+\mathrm{div}(-\mathfrak h_\Gamma)+rv_0$. We claim that $\mathrm{slp}_{e,v_0}(\mathfrak h_\Gamma)\leq 0$ for all $e\in E(G)$ incident on $v_0$. 

Suppose there is a $e_0\in E(\overline G)$ incident on $v_0$ and $\tilde e_0\in E(G)$ over $e_0$ such that $\mathrm{slp}_{\tilde e_0,v_0}(\mathfrak h_\Gamma)>0$ (or equivalently $\mathrm{slp}_{\tilde e_0,v_0}(-\mathfrak h_\Gamma)<0$). Let $v_1$ be the other vertex of $e_0$. Since $\mathcal D'_\Gamma$ is $v_0$ reduced, for all $x\in \Gamma\backslash V(G)$ we have $\mathrm{ord}_x(-\mathfrak h_\Gamma)\leq 1$. Therefore for all $\tilde e\in E(G)$ over $e_0$ either $\mathrm{slp}_{\tilde e,v_1}(-\mathfrak h_\Gamma)>0$ or there is a point $x\in \bar e^\circ$ such that $\mathrm{ord}_{x}(-\mathfrak h_\Gamma)>0$, where $\bar e\in E(\Gamma)$ is the edge induced by $\tilde e$. Now the Dhar's algorithm (\cite{luo2011rank}) implies that there must exist an edge $e_1\in E(\overline G)\backslash \{e_0\}$ incident on $v_1$ and $\tilde e_1\in E(G)$ over $e_1$ such that $\mathrm{slp}_{\tilde e_1,v_1}(-\mathfrak h_\Gamma)<0$. Inductively we end up with a sequence of distinct vertices $v_0,v_1,...$ and edges $\tilde e_1,\tilde e_2,...\subset E(G)$ over $e_1, e_2,...\subset E(\overline G)$ respectively such that $e_i$ is incident on $v_{i}$ and $v_{i+1}$ and that $\mathrm{slp}_{\tilde e_i,v_i}(-\mathfrak h_\Gamma)<0$. This sequence must be finite since $\overline G$ is a tree, which provides a contradiction.

It follows that $\mathcal D'_{v_0}+\mathrm{div}(\mathfrak h_{v_0})-P_1-\cdots-P_r\geq 0$. Now the generality of $P_i$ implies that $H^0(Z_{v_0},\mathscr O_{Z_{v_0}}(\mathcal D'_{v_0}))\cap H_{v_0}$ has dimension at least $r+1$, hence we have $H_{v_0}\subset H^0(Z_{v_0},\mathscr O_{Z_{v_0}}(\mathcal D'_{v_0}))
\\=H^0(Z_{v_0},\mathscr O_{Z_{v_0}}(\mathcal D_{v_0}+D_{v_0}^{\mathrm{red}}))$.
\end{proof}

%Note that $(w_v^{\mathrm{red}})_v$ is a tight tuple of admissible multidegrees contained in $V(G(w_0))$. 
Let $\mathfrak G^r_d(\mathfrak C_{X_0,\bm n}, w_0)$ be the set of limit $\frak g^r_d$s of multidegree $w_0$. We first show the bijectivity of $\mathfrak F$ for the tuple $(w_v^{\mathrm{red}})_v$:
\begin{thm}\label{lifting theorem}
%Let $(\mathcal D,\mathcal H)$ be a limit $\mathfrak g^r_d$ on $\mathfrak C_{X_0,\bm n}$ such that $\mathcal D$ is edge-reduced and integral of multidegree $w_0$. Choose the concentrated multidegrees $(w_v^{\mathrm{red}})_v$ as above.Then there is a pre-limit $\mathfrak g^r_d$ on $(X_0,\bm n)$ with respect to the $w_0$ induced by $\mathcal D_\Gamma$ that maps to $(\mathcal D,\mathcal H)$ under the map $\mathfrak F$. Let $(w_0,(w_v^{\mathrm{red}})_v)$ be as above, then the induced 
The map $\mathfrak F_{(w_v^{\mathrm{red}})_v}$ is a bijection onto $\mathfrak G^r_d(\mathfrak C_{X_0,\bm n}, w_0)$.
\end{thm} 
\begin{proof}
It is easy to check that $\mathfrak F$ is injective. We next show the surjectivity. Take $(\mathcal D,\mathcal H)\in \mathfrak G^r_d(\mathfrak C_{X_0,\bm n})$. According to Lemma \ref{reduced multidegree} we have $H_v\subset H^0(Z_v,\mathscr O_{Z_v}(\mathcal D_v+D_v^{\mathrm{red}}))$.
 We next show that $(\mathscr L_v,V_v)_v=(\mathcal O(\mathcal D_v+D^{\mathrm{red}}_v),H_v)_v$ is a pre-limit $\mathfrak g^r_d$ on $(X_0,\bm n)$ by extending the method of the proof of \cite[Theorem 5.4]{amini2015linear}. Note that afterwards straightforward calculation shows that $\mathfrak F((\mathscr L_v,V_v)_v)=(\mathcal D,\mathcal H)$. 

We may assume that $X_0$ is a binary curve, i.e. $|V(G)|=2$ and $\mathcal D_\Gamma=D_{w_0}$. Let $e_1,...,e_m$ be the edges of $G$ over $e\in E(\overline G)$ and $v_1, v_2$ be the vertices of $e$. For simplicity let $D_j=D_{v_j}^{\mathrm{red}},H_j=H_{v_j},Z_j=Z_{v_j},\mathscr L_j=\mathscr L_{v_j}$ and $w_j=w_{v_j}^{\mathrm{red}}$ for $j=1,2$. Denote $P^j_i=P_{e_i}^{v_j}$ the points on $Z_j$ corresponding to $e_i$. Suppose $G$ is directed by $v_1\rightarrow v_2$ and $w_0=(w_G,\mu)$. Denote $x_i=\mu (e_i)$ for $1\leq i\leq m$.

Suppose $w_1$ is obtained from $w_0$ by twisting $\lambda $ times at $(e,v_2)$ and $w_2$ be obtained from $w_0$ by twisting $b-\lambda$ times at $(e,v_1)$. Then $w_2$ is equal to $w_1$ twisting $b$ times at $(e,v_1)$. 
%and we have $\deg{\mathscr L_j}=d_j$ where $d_j$ is the degree of $w_j$ at $v_j$. 
It remains to show that the condition about multivanishing sequences (Condition (I) of Definition \ref{limit linear series on curves}) is satisfied.

Let $a_0^{e,v_j},...,a_r^{e,v_j}$ be the multivanishing sequences of $H_j$ along $D_\bullet^{e,v_j}$ as in Definition \ref{limit linear series on curves}.
 For $l=0,1,...,r$ we claim that there is an $s$ such that $a^{e,v_1}_l\geq D^{e,v_1}_s$ and $a^{e,v_2}_{r-l}\geq D^{e,v_2}_{b-s}$. 

Take effective divisors $(E_1,E_2)=(\sum_{1\leq k\leq r-l}Q_k, \sum_{r-l+1\leq k\leq r}Q_k)\in (Z_1^{(r-l)},Z_2^{(l)})$ where $Z_j^{(l)}$ is the $l$-fold symmetric product.
Let 
$$\begin{array}{ll}F_{E_1,E_2}&=\{ \mathcal E\in \mathrm{Div}(\mathfrak C_{X_0,\bm n})\colon \mathrm{there\ exists\ a\ rational\ function\ }\mathfrak f\mathrm{\ on\ }\mathfrak C\mathrm{\ such\ that\ } \\&\mathfrak f_{v_j}\in H_j\mathrm{\ and\ }\mathcal E=\sum\sum \mathrm{slp}_{e_i,v_j} (\mathfrak f_\Gamma)P_i^j \mathrm{\ and\ }\mathrm{div}(\mathfrak  f)+\mathcal D-(E_1+E_2)\geq 0\}.\end{array}$$ Then the set 
$$F=\bigcup_{(E_1,E_2)}F_{E_1,E_2}$$
 is finite. For $D_0\in F$, let $S_{D_0}$ be the subset of $Z_1^{(r-l)}\times Z_2^{(l)}$ defined by all collections of effective divisors $(E_1,E_2)$ such that $D_0\in F_{E_1,E_2}$. We have:

(i) Each $S_{D_0}$ is Zariski closed in $Z_1^{(r-l)}\times Z_2^{(l)}$. 

(ii) $\bigcup_{D_0\in F} S_{D_0}=Z_1^{(r-l)}\times Z_2^{(l)}$. 

Therefore we can find a $D_0$ such that $S_{D_0}=Z_1^{(r-l)}\times Z_2^{(l)}$, and a rational function $\mathfrak f$ on $\mathfrak C$ such that $D_0=\sum\sum \mathrm{slp}_{e_i,v_j}(\mathfrak f_\Gamma)P_i^j$.

Now for any choices of $Q_1,...,Q_{r-l}\in Z_1$ and $Q_{r-l+1},...,Q_r\in Z_2$, setting $I_1=\{1,...,r-l\}$ and $I_2=\{r-l+1,...,r\}$ and $\lambda_i^j=-\mathrm{slp}_{e_i,v_j}(\mathfrak f_\Gamma)$ we have 
$$\mathcal D_j+D_j-(D_j+\sum_i \lambda_i^jP^j_i)+\mathrm{div}(\mathfrak f_{v_j})-\sum_{k\in I_j} Q_k\geq 0$$
for $j=1,2$ and 
%$$\mathcal D_j+\sum_i \mathrm{slp}_{e_i,v_j}(\mathfrak f_\Gamma)P^j_i+\mathrm{div}(\mathfrak f_{v_j})-\sum_{k\in I_j} Q_k\geq 0$$
\begin{equation}(\mathrm{div}(\mathfrak f_\Gamma)+\mathcal D_\Gamma)|_{\Gamma\backslash V(G)}\geq 0. \end{equation}

%Denote $\lambda_i^j=-\mathrm{slp}_{e_i,v_j}(\mathfrak f_\Gamma)$. We then have:
It follows that 
$$\dim(H_1\cap H^0(Z_1,\mathscr L_1(-D_1-\sum \lambda_i^1P^1_i))\geq r-l+1$$ and that $$ \dim(H_2\cap H^0(Z_2,\mathscr L_2(-D_2-\sum \lambda_i^2P^2_i))\geq l+1.$$
We claim that there is an $a$ such that $D_1+\sum \lambda_i^1P^1_i\geq D_a^{e,v_1}$ and $D_2+\sum\lambda_i^2P^2_i\geq D_{b-a}^{e,v_2}$, and hence $a_{l}^{e,v_1}\geq \deg D_a^{e,v_1}$ and $a_{r-l}^{e,v_2}\geq \deg D_{b-a}^{e,v_2}$ and the vanishing condition is satisfied.

Note that $D_1=D^{e,v_1}_\lambda$ and $D_2=D^{e,v_2}_{b-\lambda}$. Let $n_i=\bm n(e_i)$, let $F_i\colon [0,n_i]\rightarrow \mathbb R$ be the piecewise linear function such that $F_i=\mathfrak f_\Gamma|_{e_i}$. With out lost of generality assume $F_i(0)=0$ and $F_i(n_i)=y\geq 0$. We have $\lambda_i^1=-F_i'(0)$ and $\lambda_i^2=F'_i(n_i)$. Take $\widetilde F_i\colon [0,n_i]\rightarrow \mathbb R$ such that $\widetilde F_i(x)=0$ if $0\leq x\leq x_i$ and $\widetilde F_i(x)=x-x_i$ otherwise.
According to $(1)$, $F_i(x)$ is convex when $x_i=0$ and $F_i(x)+\widetilde F_i(x)$ is convex when $x_i\neq 0$. We then have $\lambda_i^1\geq-\lfloor \frac{y}{n_i}\rfloor$ and $\lambda_i^2\geq \lceil\frac{y}{n_i}\rceil$ when $x_i= 0$, and $\lambda^1_i\geq -\lfloor\frac{y+n_i-x_i}{n_i}\rfloor$ and $\lambda_i^2\geq\lceil\frac{y-x_i}{n_i}\rceil$ when $x_i\neq 0$.

%$\lambda_i^1\geq-\lfloor \frac{y+n_i-x_i}{n_i}\rfloor$ and $\lambda_i^2\geq \lceil\frac{y+n_i-x_i}{n_i}\rceil$ when $x_i\neq 0$. 
Now it is enough to find $a$ such that:
$$\sum_{i\colon x_i=0}\lfloor \frac{y}{n_i}\rfloor P_i^1 +\sum_{i\colon x_i\neq 0}\lfloor \frac{y+n_i-x_i}{n_i}\rfloor P_i^1\leq D_\lambda^{e,v_1}-D_a^{e,v_1}=\sum_{k=a}^{\lambda-1}\ \ \sum_{i\colon x_i-\lambda+k\equiv 0(\mathrm{mod}n_i)}P^1_i$$ and that 
$$\sum\lceil \frac{y-x_i}{n_i}\rceil P_i^2\geq D_{b-a}^{e,v_2}-D_{b-\lambda}^{e,v_2}=\sum_{k=b-\lambda}^{b-a-1}\ \ \sum_{i\colon x_i-\lambda+b-k\equiv 0(\mathrm{mod}n_i)}P^2_i.$$

In other words we must have $$\lfloor \frac{y}{n_i}\rfloor\leq\sum_{\tiny{\begin{array}{cc}k\colon a\leq k\leq \lambda-1\\x_i-\lambda+k\equiv 0(\mathrm{mod}n_i)\end{array}}}1 \mathrm{\ \ and\ \  }\lceil \frac{y}{n_i}\rceil\geq\sum_{\tiny{\begin{array}{cc}k\colon b-\lambda\leq k\leq b-a-1\\x_i-\lambda+b-k\equiv 0(\mathrm{mod}n_i)\end{array}}} 1\mathrm{\ \ if\ \ } x_i=0$$
and
 $$\lfloor \frac{y+n_i-x_i}{n_i}\rfloor\leq\sum_{\tiny{\begin{array}{cc}k\colon a\leq k\leq \lambda-1\\x_i-\lambda+k\equiv 0(\mathrm{mod}n_i)\end{array}}}1 \mathrm{\ \ and\ \  }\lceil \frac{y-x_i}{n_i}\rceil\geq\sum_{\tiny{\begin{array}{cc}k\colon b-\lambda\leq k\leq b-a-1\\x_i-\lambda+b-k\equiv 0(\mathrm{mod}n_i)\end{array}}} 1\mathrm{\ \ if\ \ } x_i\neq 0.$$
It is straightforward to verify that $a=\lambda-\lfloor y\rfloor$ satisfies the inequalities above.
%For an arbitrary curve $X_0$ of pseudocompact type the argument above still works, as long as there exists concentrated multidegrees $\{w_v\}_v$ such that for any adjacent $v,v'\in\overline{G}$ we have $w_v$ obtained from $w_{v'}$ by twisting $b_{v,v'}$ times at $(e,v')$, where $e $ is the edge in $\overline G$ that connects $v$ and $v'$; and that the corresponding divisor on $Z_v$ obtained by the twistings from $w_0$ to $w_v$ has sufficiently big coefficients at its support. This is easy to verify since $\Gamma$ is a multitree.
\end{proof}

%\begin{rem}
%$$\sum_{k\colon x_i-(\lambda-a)\leq k\leq x_i-1, k\equiv 0(\mathrm{mod} n_i)}1\mathrm{\ \ and\ \ } 
%\sum_{k\colon x_i-(\lambda-a)+1\leq k\leq x_i, k\equiv 0(\mathrm{mod} n_i)}1$$
%\end{rem}
 
We next show that there is an isomorphism between $P^r_d(X_0,\bm n,(w_v)_v)$ and $ P^r_d(X_0,\bm n,(w^{\mathrm{red}}_v)_v)$ that is compatible with $\mathfrak F_{(w_v)_v}$ and $\mathfrak F_{(w^{\mathrm{red}}_v)_v}$ for any tight tuple $(w_v)_v$, and the bijectivity of $\mathfrak F_{(w_v)_v}$ then follows.

\begin{cor}\label{pre lls isomorphism}
For any tight tuple $(w_v)_v$ of admissible multidegrees in $V(G(w_0))$ we have a natural isomorphism $\mathfrak I\colon P^r_d(X_0,\bm n,(w_v)_v)\rightarrow P^r_d(X_0,\bm n,(w^{\mathrm{red}}_v)_v)$ such that $\mathfrak F_{(w_v)_v}=\mathfrak F_{(w^{\mathrm{red}}_v)_v}\circ \mathfrak I$. Moreover, $\mathfrak I$ induces a bijection between the sets of pre-limit linear series that satisfy the weak glueing condition. 
\end{cor}
\begin{proof}
Since $\mathfrak F_{(w^{\mathrm{red}}_v)_v}$ is a bijection, we define $\mathfrak I=\mathfrak F_{(w^{\mathrm{red}}_v)_v}^{-1}\circ \mathfrak F_{w_0,(w_v)_v}$. To show that this is an isomorphism, we construct an inverse $\mathfrak L$ of $\mathfrak I$. 

We first claim for all $v\in V(G)$ that $w_v$ is obtained from $w_v^{\mathrm{red}}$ by twisting vertices in $V(G)\backslash \{v\}$. Suppose there is a path $P(w_v^{\mathrm{red}},v_1,...,v_m)$ in $G(w_0)$ from $w_v^{\mathrm{red}}$ to $w_v$ such that $ \{v_i\}_{1\leq i\leq m}\subsetneq V(G)$.
%Suppose $w_v$ is equal to $w_v^{\mathrm{red}}$ twisting consecutively at $v_1,...,v_m$ (not necessarily distinct) such that $v\in \{v_i\}_{1\leq i\leq m}\subsetneq V(G)$. 
Take an ordering $v=v_0',v_1',v_2',...$ of $V(G)$ such that $w_v$ becomes negative in $v_i'$ for $i\geq 1$ after taking the negative twists of $v_0',...,v_{i-1}'$. Let $j$ be the smallest number such that $v_j'\not\in \{v_i\}_{1\leq i\leq m}$. Then negatively twisting $w_v$ at $v_0',...,v_{j-1}'$  is the same as twisting $w_v^{\mathrm{red}}$ at vertices in $\{v_i\}_{1\leq i\leq m}\backslash \{v_0',...,v_{j-1}'\}$ (perhaps with multiplicities), which can not be negative at $v_j'$ unless $j=0$, namely $v_j=v$, since $w_v^{\mathrm{red}}$ is nonnegative at all vertices except $v$. So $v\not\in \{v_i\}_{1\leq i\leq m}$.

It follows that $D_v=D_{w_v^{{\mathrm{red}}},w_v}^v$ is effective for all $v$. Let $(\mathscr L_v,V_v)_v\in P^r_d(X_0,\bm n,(w^{\mathrm{red}}_v)_v)$. We define $\mathfrak L((\mathscr L_v,V_v)_v)=(\mathscr L_v',V_v')_v$ where $\mathscr L_v'=\mathscr L_v(D_v)$ and $V_v'$ is the image of $V_v$ under the imbedding $i_v\colon H^0(Z_v,\mathscr L_v) \rightarrow H^0(Z_v,\mathscr L_v')$. We next show that $(\mathscr L_v',V_v')_v$ is contained in $P^r_d(X_0,\bm n,(w_v)_v)$ and satisfies the weak glueing condition if and only if $(\mathscr L_v,V_v)_v$ does.

Take $e\in \overline G$ incident on $v$ and $v'$. Suppose $w_v^{\mathrm{red}}$ (resp. $w_v$) is obtained from $w_{v'}^{\mathrm{red}}$ (resp. $w_{v'}$) by twisting $b_{v,v'}$ (resp. $b_{v,v'}'$) times at $(e,v')$. Let $D_i^{e,v}$ (resp. $(D_i^{e,v})'$) be the twisting divisors associated to $(w_v^{\mathrm{red}})_v$ (resp. $(w_v)_v$). 
Let $(a_l^{e,v})_l$ (resp.$((a_l^{e,v})')_l$) be the multivanishing sequence of $V_v$ (resp. $V_v'$) along $D_\bullet^{e,v}$ (resp. $(D_\bullet^{e,v})'$). 

Let $v_1,...,v_m$ be as above. Take also a path $P(w_{v'}^{\mathrm{red}},v_1',...,v_{m'}')$ in $G(w_0)$ from $w_{v'}^{\mathrm{red}}$ to $w_{v'}$ such that $v'\not\in\{v_i'\}_{1\leq i\leq m'}$. Suppose $v$ appears $a$ times in $v_1',...,v_{m'}'$ and $v'$ appears $a'$ times in $v_1,...,v_m$. One easily verifies that $b'_{v_1,v_2}=b_{v_1,v_2}+a'+a$; and that if $a_l^{e,v}=\deg D_j^{e,v}$ for some $j$ critical, then $(a_l^{e,v})'=\deg ((D_{j+a'}^{e,v})')$ with $j+a'$ critical; and if $a_l^{e,v'}=\deg D_j^{e,v'}$ for some $j$ critical, then $(a_l^{e,v'})'=\deg ((D_{j+a}^{e,v'})')$ with $j+a$ critical. It then follows by definition that $(\mathscr L_v',V_v')_v$ is in $P^r_d(X_0,\bm n,(w_v)_v)$.

Moreover, for $s\in V_{v}$ and $s'=i_v(s)\in V_v'$, we have $\mathrm{ord}^0_{P}(s)-\mathrm{ord}_P({D_\bullet^{e,v}})=\mathrm{ord}^0_{P}(s')-\mathrm{ord}_P({(D^{e,v}_{\bullet+a'})'})$ for $P\in\mathcal A_e^v$, and the same holds if we replace $v$ by $v'$ and $a'$ by $a$. In particular, if $W_{v}\subset V_{v}(-D_j^{e,v})/V_{v}(-D_{j+1}^{e,v})$ and $W_{v'}\subset V_{v'}(-D_{b_{v,v'}-j}^{e,v'})/V_{v'}(-D_{b_{v,v'}-j+1}^{e,v'})$ as in Definition \ref{weak glueing condition} satisfies the weak glueing condition for $(w^{\mathrm{red}}_v)_v$ then we can take $W_{v}'=i_v(W_v)\subset V_v'$ and $W_{v'}'=i_{v'}(W_{v'})\subset V_{v'}'$ which fulfills the weak glueing condition for $(w_v)_v$, and vice versa. 
\end{proof}

\begin{cor}\label{isomorphism for tight tuple}
For any tight tuple $(w_v)_v$ in $V(G(w_0))$ the map $\mathfrak F_{(w_v)_v}\colon P^r_d(X_0,\bm n,(w_v)_v)\rightarrow \mathfrak G^r_d(\mathfrak C_{X_0,\bm n},w_0)$ is a bijection, the preimage of $(\mathcal D,\mathcal H)$ with $\mathcal D_\Gamma=D_{w_0}$ is $(\mathscr O_{Z_v}(\mathcal D_v+D^v_{w_0,w_v}),H_v)_v$. In particular we have $H_v\subset H^0(Z_v, \mathscr O_{Z_v}(\mathcal D_v+D^v_{w_0,w_v}))$.
\end{cor}

We are now able to define the weak glueing condition for limit linear series on $\mathfrak C_{X_0,\bm n}$ of multidegree $w_0$ by lifting to a pre-limit linear series on $(X_0,\bm n)$, and the definition is independent of the choice of lifting by Corollary \ref{pre lls isomorphism}:

\begin{defn}\label{weak glueing revised}
Let $(\mathcal D,\mathcal H)\in\mathfrak G^r_d(\mathfrak C_{X_0,\bm n},w_0)$ be a limit $\mathfrak g^r_d$ on $\mathfrak C_{X_0,\bm n}$. We say that $(\mathcal D,\mathcal H)$ satisfies the \textbf{weak glueing condition} with respect to $w_0$ if $\mathfrak F_{(w_v)_v}^{-1}((\mathcal D,\mathcal H))$ satisfies the weak glueing condition with respect to $(w_v)_v$ for a tight tuple $(w_v)_v$ in $V(G(w_0))$. 
\end{defn}

%The weak glueing condition above is independent of $w_0$ and invariant under scaling of metric graphs and multidegrees: 

\begin{prop}\label{weak glueing well behaved}
%$\mathrm{(1)}$ Let $(w_v)_v$ and $(w_v')_v$ be concentrated degrees as in above definition. Then $(\mathcal D,\mathcal H)$ satisfies the weak glueing condition with respect to $(w_v)_v$ if and only if it satisfies the weak glueing condition with respect to $(w_v')_v$.
%$\mathrm{(1)}$ If $(\mathcal D,\mathcal H)$ and $(\mathcal D',\mathcal H')$ are rational equivalent divisors on $\mathfrak C_{X_0,\bm n}$. Then $(\mathcal D,\mathcal H)$ satisfies the weak glueing condition if and only if $(\mathcal D',\mathcal H')$ satisfies the weak glueing condition.
$\mathrm{(1)}$ Suppose $(\mathcal D,\mathcal H)\in\mathfrak G^r_d(\mathfrak C_{X_0,\bm n},w)$ for $w=w_0$ and $w=w_0'$. Then $(\mathcal D,\mathcal H)$ satisfies the weak glueing condition with respect to $w_0$ if and only if it satisfies that with respect to $w_0'$.

$\mathrm{(2)}$ Suppose $w_0=(w_G,\mu)$. Take $m\in \mathbb Z_{\geq 0}$ and $\bm n'\colon E(G)\rightarrow \mathbb Z_{>0}$ such that $\bm n'(e)=m\bm n(e)$ for all $e\in E(G)$. Let $(\mathcal D',\mathcal H')$ be a limit $\mathfrak g^r_d$ on $\mathfrak C_{X_0,\bm n'}$ such that $ H_v'= H_v$ and $\mathcal D'$ corresponds to $w'=(w_G',\mu')$ where $w_G'=w_G$ and $\mu'(e)=m\mu(e)$ for all $e\in E(G)$. Then $(\mathcal D,\mathcal H)$ satisfies the weak glueing condition if and only if $(\mathcal D',\mathcal H')$ satisfies the weak glueing condition.
\end{prop} 
\begin{proof}
(1) Note that Definition \ref{weak glueing revised} only depends on $V(G(w_0))$. We have $D_{w_0}$ linearly equivalent to $D_{w_0'}$ as divisors on $\Gamma$. It is then easy to check that $w_0\in V(G(w_0'))$ and hence $G(w_0)=G(w_0')$. 

(2) Obvious, note that the scaling of a concentrated admissible multidegree is still concentrated. 
\end{proof}

Now we define the weak glueing condition for arbitrary limit linear series on $\mathfrak C_{X_0,\bm n}$:

\begin{defn}\label{weak glueing completed}
We say that a limit linear series $(\mathcal D,\mathcal H)$ on $\mathfrak C_{X_0,\bm n}$ satisfies the \textbf{weak glueing condition} if after scaling (by integer) of $\bm n$ there is a multidegree $w_0$ on $(G,\bm n)$ such that $(\mathcal D,\mathcal H)$ is contained in $ \mathfrak G^r_d(\mathfrak C_{X_0,\bm n},w_0)$ and satisfies the weak glueing condition with respect to $w_0$.

%we have $(\mathcal D,\mathcal H)$ rational equivalent to some $(\mathcal D',\mathcal H')$ such that $\mathcal D'$ integral and edge-reduced of multidegree $w_0$ and that $(\mathcal D',\mathcal H')$ satisfies the weak glueing condition.
\end{defn}
Note that the definition above is independent of the choice of $w_0$ and scaling by Proposition \ref{weak glueing well behaved}. See Example \ref{example} for a imit linear series on a metrized complex that does not satisfy the weak glueing condition.

\section{Smoothing of limit linear series on metrized complexes} 
In this section we use the same symbols as in Notation \ref{notation} and assume that $\kappa$ is algebraically closed. We consider special cases of $w_0$ and $(X_0,\bm n)$ such that (1) the space of limit linear series on $(X_0,\bm n)$ is of expected dimension, and (2) any pre-limit linear series satisfying the weak glueing condition lifts to a limit linear series on $(X_0,\bm n)$. This along with the forgetful map constructed in the last section will guarantee the smoothability of certain limit linear series on $\mathfrak C_{X_0,\bm n}$. 

In order to prove the main theorems, we first recall some results from \cite{osserman2014limit} and \cite{osserman2014dimension} for convenience. 

\begin{thm}\label{a smoothing theorem} 
$\mathrm{(}$\cite[Theorem 6.1]{osserman2014limit}$ \mathrm{)}$ If $\pi\colon X\rightarrow \mathrm{Spec}(R)$ is a regular smoothing family with special fiber $\widetilde X_0$, and the space of limit linear series on $(X_0,\bm n)$ has the expected dimension $\rho=g+(r+1)(d-r-g)$, then every limit linear series on $(X_0,\bm n)$ can be smoothed to a linear series on the generic fiber $X_\eta$. 
\end{thm}

\begin{defn}
Let $X$ be a smooth projective curve of genus $g$, and fix integers $r,d,n>0$, and for $i=1,...,n$ fix also $m_i>0$. Choose distinct points $P_i^k$ on $X$ for $i=1,...,n$ and $k=1,...,m_i$. Then we say that $(X,(P_i^k)_{i,k})$ is \textbf{strongly Brill-Noether general} for $r,d$ if, for all tuples of nondecreasing effective divisor sequences $D^i_\bullet$, such that every divisor $D^i_\bullet$ is supported on $P_i^1,...,P^{m_i}_i$, and for every tuple of nondecreasing sequences $a_\bullet^i$ such that for each $j$ we have $a_j^i=\deg D^i_l$ for some $\ell$ critical for $D^i_\bullet$ and the number of repetitions of $a^i_j$ is at most $\deg (D_{\ell+1}^i-D_\ell^i)$, the space $G^r_d(X,(D^i_\bullet,a^i_\bullet)_i)$ of $\mathfrak g^r_d$s on $X$ having multivanishing sequence at least $a^i_\bullet$ along $D^i_\bullet$ for each $i$ has the expected dimension
$$\rho:=g+(r+1)(d-r-g)-\sum_{i=1}^{n}\Biggl(\sum_{j=0}^{r}(a_j^i-j)+\sum_{\ell=0}^{b_i}{r^i_\ell\choose 2}\Biggr)$$ if it is nonempty. In the above, $D^i_\bullet$ is indexed from $0$ to $b_i+1$, and $r_\ell^i$ is defined to be $0$ if $\ell$ is not critical for $D^i_\bullet$, and the number of times $\deg D_\ell^i$ occurs in $a^i_\bullet$ if $\ell$ is critical.
\end{defn}

We refer to \cite[Theorem 3.3]{osserman2014dimension} for examples of strongly Brill-Boether general curves. The simplest example is that $X$ is a rational curve and $n\leq 2$, which will be used in Section 5.

\begin{thm}\label{theorem of dimension}
$\mathrm{(}$\cite[Corollary 5.2]{osserman2014dimension}$ \mathrm{)}$ 
Suppose $(X_0,\bm n)$ satisfies the following condition:

$\ \ \ \mathrm{(I)}$ there are at most three edges of $G$ connecting any given pair of vertices;

$\ \mathrm{(II)}$ there exists a $d'$ such that for any adjacent vertices $v, v'$ of $G$, connected by edges $(e_i)_i$, and any integers $(x_i)_i$ with $\sum_i x_i\bm n(e_i)=0$, if there is a unique $j$ with $x_j>0$ then we have $\sum_i\lfloor x_j\bm n(e_j)/\bm n(e_i)\rfloor>d'$;

$\mathrm{(III)}$ each marked component of $X_0$ is strongly Brill-Noether general.

Then the space of limit linear series on $(X_0,\bm n)$ (of multidegree $w_0$) of degree $d$ with $d\leq d'$ is pure of expected dimension $\rho$.
\end{thm}

We first let $\bm n$ vary. For a family of $w_0$ the condition (1) and (2) in the beginning of this section is satisfied for all $X_0$ with strongly Brill-Noether general components. Moreover the weak glueing condition is satisfied automatically.   

\begin{thm}\label{smoothing general limit grd}
Let $\{e^i_j\}_{1\leq i\leq k_j}$ be all edges of $G$ over an edge $e_j\in E(\overline G)=\{e_j\}_{j\in J}$. Let $(\mathcal D,\mathcal H)$ be a limit $\mathfrak g^r_d$ on $\mathfrak C_{X_0,\bm n}$ with rational and edge-reduced $\mathcal D$. Suppose $G$ is directed such that the edges over the same edge of $\overline G$ have the same tail. Let $x_j^i\in[0, \bm n(e_j^i))$ correspond to the point in $\mathcal D\cap (e_j^i)^\circ$ where we set $x_j^i=0$ if the intersection is empty. Suppose that:

$\ \mathrm{(1)}$ There is a number $m\in \mathbb Z$ such that $mx_j^i\in \mathbb Z$ for all $j$ and $1\leq i\leq k_j$, and for each $j$ the classes of $mx_j^i$ where $1\leq i\leq k_j$ modulo gcd$(m\bm n(e_j^1),...,m\bm n(e_j^{k_j}))$ are distinct.

$\ \mathrm{(2)}$ Each component of $X$ is strongly Brill-Noether general.

Then the limit linear series $(\mathcal D,\mathcal H)$ is smoothable.

 \end{thm}

Note that if $\mathcal D$ is of multidegree $w_0=(w_G,\mu)$, then condition (1) of Theorem \ref{smoothing general limit grd} says that, with the given direction of $G$ in loc.cit., for each $j$ the classes of $\mu (e_j^i)$ where $1\leq i\leq k_j$ modulo gcd$(\bm n(e_j^1),...,\bm n(e_j^{k_j}))$ are distinct.

\begin{lem}\label{compatibleness between smoothing}
Let $(\mathscr L,(V_v)_v)$ be a limit $\mathfrak g^r_d$ on $(X_0,\bm n)$ of multidegree $w_0$ with respect to $(w_v)_v$, and $(\mathscr L^{v},V_v)_v$ be the induced pre-limit linear series. Let $X\rightarrow \mathrm{Spec}(R)$ be a regular smoothing family with special fiber $\widetilde X_0$. If $(\mathscr L,(V_v)_v)$ can be smoothed to a $\mathfrak g^r_d$ on $X_\eta$ then $\mathfrak F((\mathscr L^v,V_v)_v)$ is smoothable in Amini-Baker sense. 
\end{lem} 
\begin{proof}
Follows directly from \cite[Proposition 4.16]{osserman2017limit}. 
\end{proof}

\begin{proof}[Proof of Theorem \ref{smoothing general limit grd}]Up to scaling we may assume that $\mathcal D$ integral and of multidegree $w_0$, hence we can take $m=1$. Take a pre-limit linear series $(\mathscr L_v, V_v)_v$ with respect to $(w_v)_v$ that maps to $(\mathcal D,\mathcal H)$ under $\mathfrak F_{(w_v)_v}$. Then condition (1) implies that $\deg(D_{j+1}^{e,v})-\deg(D_j^{e,v})=1$, and the morphism $\varphi_j^{(e,v)}$ in Definition \ref{limit linear series on curves} is a linear morphism between one dimensional spaces. Hence for any line bundle $\mathscr L$ on $\widetilde X_0$ of multidegree $w_0$ such that $\mathscr L^v=\mathscr L_v$ and any enriched structure on $\widetilde X_0$ the condition (II) of loc.cit. is satisfied automatically. In other words $(\mathscr L, (V_v)_v)$ is a limit linear series. By Lemma \ref{compatibleness between smoothing}, it remains to show that $(\mathscr L, (V_v)_v)$ is smoothatble. According to \cite[Theorem B.2]{baker2008specialization} there is a regular smoothing family $X$ over $B$ with special fiber $\widetilde X_0$, and Theorem \ref{a smoothing theorem} says that $(\mathscr L, (V_v)_v)$ can be smoothed to a $\mathfrak g^r_d$ on $X_\eta$ if the the space $G^r_{ w_0}(X_0,\bm n,(\mathscr O_v)_v)$ of limit $\mathfrak g^r_d$s on $(X_0,\bm n)$ has dimension $\rho$ where $(\mathscr O_v)_v$ is the enriched structure induced by $X$. This is ensured by condition (2) and \cite[Theorem 4.1]{osserman2014dimension} since there is no glueing condition imposed. 
\end{proof}

Condition (1) of Theorem \ref{smoothing general limit grd} is trivial when $X_0$ is of compact type, whereas condition (2) is still necessary, see \cite[Example 5.14]{amini2015linear} for a counterexample. As a result we have the following Corollary:

\begin{cor} 
Let $X_0$ be a curve of compact type, and assume the marked components of $X_0$ are all strongly Brill-Noether general. Let $(\mathcal D,\mathcal H)$ be a limit linear series on the metrized complex $\mathfrak C_{X_0,\bm n}$ with rational $\mathcal D$. Then $(\mathcal D,\mathcal H)$ is smoothable. 
\end{cor}

We next consider $(X_0,\bm n)$ as in Theorem \ref{theorem of dimension}, and prove that the weak glueing condition is a  necessary and sufficient condition for the smoothability of a limit linear series on $\mathfrak C_{X_0,\bm n}$ (of certain degree). The necessity is actually valid for arbitrary $(X_0,\bm n)$:

\begin{thm}\label{smoothness implies weak glueing} 
Suppose $(\mathcal D,\mathcal H)$ is a smoothable limit linear $\mathfrak g^r_d$ on $\mathfrak C_{X_0,\bm n}$. Then $(\mathcal D,\mathcal H)$ satisfies the weak glueing condition.
\end{thm} 
\begin{proof}
Since $\mathcal D_\Gamma$ is linearly equivalent to a rational divisor on $\Gamma$, we may assume that $\mathcal D$ is integral and edge-reduced and $\mathcal D_\Gamma=D_{w_0}$. Suppose there is a strongly semistable model $\mathfrak X$ over $\widetilde R$ with associated metrized graph $\mathfrak C\mathfrak X=\mathfrak C_{X_0,\bm n}$ and a linear series $(D_\eta,H_\eta)$ on $\mathfrak X_\eta$ of rank $r$ that specializes to $(\mathcal D,\mathcal H)$.  

Suppose the concentrated degrees $(w_v)_v=(w_v^{\mathrm{red}})_v$ are given. Let $D_v=D^v_{w_0,w_v^{\mathrm{red}}}$. We have $H_v\subset H^0(Z_v,\mathscr O_{Z_v}(\mathcal D_v+D_v))$. Consider $v$ and $v'$ in $V(\overline G)$ connected by $e\in E(\overline G)$. Let $\{e_i\}\subset E(G)$ be the edges over $e$, let $P_i=P_{e_i}^v$ and $P_i'=P_{e_i}^{v'}$. Take $h\in H_v\subset H^0(Z_v,\mathscr O_{Z_v}(\mathcal D_v+D_v))$. 
Choose $g\in H_\eta$ such that $\tau_{*v}^{\mathfrak C\mathfrak X}(g)=h$ up to multiplication by $\kappa^\times$. Let $h'=\tau_{*v'}^{\mathfrak C\mathfrak X}(g).$ According to the Slope Formula \cite[Theorem 4.4(3)]{amini2015linear} we have 
\begin{equation}\mathrm{ord}_{P_i}(h)=-\mathrm{slp}_{e_i,v}(\tau_*^{\mathfrak C\mathfrak X}(g)_\Gamma)\mathrm{\ \ and\ \ } \mathrm{ord}_{P_i'}(h')=-\mathrm{slp}_{e_i,v'}(\tau_*^{\mathfrak C\mathfrak X}(g)_\Gamma).\end{equation} 

On the other hand, since $D_\eta+\mathrm{div}(g)\geq 0$, by  \cite[Theorem 4.5]{amini2015linear} we have $$0\leq \tau_*^{\mathfrak C\mathfrak X}(D_\eta+\mathrm{div}(g))=\mathcal D+\mathrm{div}(\tau_*^{\mathfrak C\mathfrak X}(g)).$$
Restricting to $e_i^\circ$, where $e_i$ is considered as an edge of $\Gamma$ (with vertex set $V(G)$), we have \begin{equation}\mathcal D|_{e_i^\circ}+\mathrm{div}(\tau_*^{\mathfrak C\mathfrak X}(g)_\Gamma)|_{e_i^\circ}\geq 0.\end{equation}

%The only case we need to consider is when $\deg(D_{l+1}^{e,v}-D_{l}^{e,v})=\deg(D_{b+1-l}^{e,v'}-D_{b-l}^{e,v'})\geq 2$. Without loss of generality we assume $\mathrm{ord}_{P_i}({D_{l+1}^{(e,v)}})-\mathrm{ord}_{P_i}({D_l^{(e,v)}})=1$ for all $i$. 

By (2) for each $j$ we have (note that $\mathcal D_v$ is supported on $Z_v\backslash \mathcal A_v$)
$$\mathrm{ord}^0_{P_i}(h)-\mathrm{ord}_{P_i}(D_j^{e,v})=\mathrm{slp}_{e_i,v}(-\tau_*^{\mathfrak C\mathfrak X}(g)_\Gamma)-\mathrm{ord}_{P_i}(D_j^{e,v}-D_v).$$
Let $w^j_v$ be the multidegree obtained from $w_v$ by twisting $j$ times at $(e,v)$ and $\lambda$ be the piecewise linear function on $\Gamma$ such that $\mathrm{div}(\lambda)=\mathcal D_{\Gamma}-D_{w_v^j}$. Let $F_g=-\tau_*^{\mathfrak C\mathfrak X}(g)_\Gamma-\lambda$. We then have $\mathrm{ord}_{P_i}(D_j^{e,v}-D_v)=\mathrm{slp}_{e_i,v}(\lambda)$ and 
\begin{equation}\mathrm{ord}^0_{P_i}(h)-\mathrm{ord}_{P_i}(D_j^{e,v})=\mathrm{slp}_{e_i,v}(F_g).\end{equation}

A symmetric argument shows that 
\begin{equation}\mathrm{ord}^0_{P_i'}(h')-\mathrm{ord}_{P_i'}(D_{r-j}^{e,v'})=\mathrm{slp}_{e_i,v'}(F_g).\end{equation}
 According to (3) we have 
\begin{equation}\mathrm{div}(-F_g)|_{e_i^\circ}= \mathcal D|_{e_i^\circ}+\mathrm{div}(\tau_*^{\mathfrak C\mathfrak X}(g)_\Gamma)|_{e_i^\circ}+(\mathrm{div}(\lambda)-\mathcal D)|_{e_i^\circ}\geq -D_{w_v^j}|_{e_i^\circ}.\end{equation}
%where $(\mathrm{div}(\lambda)-\mathcal D)|_{e_i^\circ}=w_v^l|_{e_i^\circ}=0$ since $\mathrm{ord}_{P_i}({D_{l+1}^{(e,v)}})-\mathrm{ord}_{P_i}({D_l^{(e,v)}})=1.$ It follows that $F_g$ is a concave function on $e_i^\circ$. 
Thus the same argument as in the proof of Theorem \ref{lifting theorem} shows that there is an $a$ such that \begin{equation}\mathrm{ord}_{D_\bullet^{e,v}}( h)\geq\deg D^{e,v'}_{a} \mathrm{\ \ and\ \ }\mathrm{ord}_{D_\bullet^{e,v'}}(h')\geq\deg D^{e,v'}_{b_{v,v'}-a}. \end{equation}

Now let $h_0,...,h_r$ be a basis of $H_v$ such that $\mathrm{ord}_{D_\bullet^{e,v}}(h_i)=a_i^{e,v}$, where $\{a_i^{e,v}\}$ is the multivanishing sequence of $(\mathscr O_{Z_v}(\mathcal D_v+D_v),H_v)$ along $D_\bullet^{e,v}$. Let $j$ 
be a critical number for $D_{\bullet}^{e,v}$. It follows that $b_{v,v'}-j$ is critical for $D_{\bullet}^{e,v'}$. Assume $H_v(-D^{e,v}_{j+1})$ is generated by $\{h_{l+1},...,h_r\}$. 
%and $H_v(-D^{e,v}_{j})$ generated by $\{h_{s+1},...,h_r\}$. 
Then $\mathrm{ord}_{D_\bullet^{e,v}}(\ell)\leq \deg(D^{e,v}_j)$ for $\ell\in\mathrm{span}(
h_0,...,h_l)$. Let $g_0,...,g_l$ be rational functions in $H_\eta$ such that $\tau_{*v}^{\mathfrak C\mathfrak X}(g_i)=h_i$. Then $L=\mathrm{span}(g_0,...,g_l)$ is an $(l+1)$-dimensional subspace of $H_\eta$ such that $\tau_{*v}^{\mathfrak C\mathfrak X}(L)=\mathrm{span}(
h_0,...,h_l)$ and $\tau_{*v'}^{\mathfrak C\mathfrak X}(L)$ is an $(l+1)$-dimensional subspace of$H_{v'}$.

It follows (7) that for any $g\in L$ we have $\tau_{*v'}^{\mathfrak C\mathfrak X}(g)\in H_{v'}(-D^{e,v'}_{b_{v,v'}-j})$. %and for any $g\in L_s$ we have $\tau_{*v'}^{\mathfrak C\mathfrak X}(g)\in H_{v'}(-D^{e,v'}_{b_{v,v'}-j'})$. Therefore we have
 %$$d=\dim( H_{v'}(-D^{e,v'}_{b_{v,v'}-j+1}))=\dim( H_{v'}(-D^{e,v'}_{b_{v,v'}-j'}))\geq s+1.$$ 
Now we can pick 
%$s_{d+1}^{e,v},...,s_l^{e,v}$ and $s_{r-d-1}^{e,v'},...,s_{r-l}^{e,v'}$ that 
the spaces $W_v$ and $W_{v'}$ (of dimension $g_j$) as described in Definition \ref{weak glueing condition}. Note that $g_j\leq l+1-\dim( H_{v'}(-D^{e,v'}_{b_{v,v'}-j+1}))$. Taking $W_{v'}$ to be an arbitrary subspace of  $\tau_{*v'}^{\mathfrak C\mathfrak X}(L)$ of dimension $g_j$ such that $W_{v'}\cap H_{v'}(-D^{e,v'}_{b_{v,v'}-j+1})=0$ and $W_v=\tau_{*v}^{\mathfrak C\mathfrak X}(W)$, where $W$ is a $g_j$ dimensional subspace of $L$ such that $\tau_{*v'}^{\mathfrak C\mathfrak X}(W)=W_{v'}$ (which can be constructed in the same way as $L$), will complete the proof. 
%$$W_v=\mathrm{span}(\overline s_{d+1}^{e,v},...,\overline s_l^{e,v})\ \mathrm{and}\ W_{v'}=\mathrm{span}(\overline s_{r-d-1}^{e,v'},...,\overline s_{r-l}^{e,v'})$$
%satisfy the condition in Definition \ref{weak glueing condition} where $\overline s^{e,v}_i\in \mathscr L_v(-D^{e,v}_j)/\mathscr L_v(-D_{j+1}^{e,v})$ and similarly for $\overline s^{e,v'}_i$, we let $s_{r-d-1}^{e,v'},...,s_{r-l}^{e,v'}\in \tau_{*v'}^{\mathfrak C\mathfrak X}(L_l)\backslash H_{v'}(-D^{e,v'}_{b_{v,v'}-j''})$ such that $\mathrm{div}(\overline s_{r-d-1}^{e,v'})|_{\mathcal A_{v'}^e}$ is strictly increasing. 
%We set $s_{i}^{e,v}$ to be the specialization of a lift of $s_{r-i}^{e,v'}$ in $L_l$ will complete the proof. 
Indeed, for all $s'\in W_{v'}$, taking $t\in W$ such that $\tau_{*v'}^{\mathfrak C\mathfrak X}(t)=s'$ and $s=\tau_{*v}^{\mathfrak C\mathfrak X}(t)\in \mathrm{span}(
h_0,...,h_l)$   
we have that $\mathrm{ord}_{D_\bullet^{e,v}}(s')=D_{b_{v,v'}-j}^{e,v'}$, and hence $\mathrm{ord}_{D_\bullet^{e,v}}(s)=D_j^{e,v}$ by (7). 
On the other hand, we have
$-D_{w_v^j}|_{e_i^\circ}=\emptyset$ and 
hence $F_{g}$ is a concave function on $e_i^\circ$ by (6) for all $i$ such that $P_i$ is in the support of $D_{j+1}-D_j$. Therefore 
for the same set of $i$ we have that $s$ vanishes at $P_i$ if and only if $s'$ vanishes at $P'_i$ by (4) and (5) where $s$ (resp. $s'$) is considered in $H_v(-D_j^{e,v})/H_v(-D_{j+1}^{e,v})$ (resp. $H_{v'}(-D_{b_{v,v'}-j}^{e,v'})/H_{v'}(-D_{b_{v,v'}-j+1}^{e,v'})$). 
Now the conclusion follows from Remark \ref{weak glueing remark}.

\end{proof}

Note that in the proof of the theorem above we do not assert that $(\mathcal D,\mathcal H)$ is the image (under $\mathfrak F_{(w_v)_v}$) of a smoothable limit linear series on $(X_0,\bm n)$, as $\mathfrak F((\mathscr L^v,V_v)_v)$ in Lemma \ref{compatibleness between smoothing}, since it's unclear whether a strongly semistable model is a base extension of a regular smoothing family. We now focus on the case of Theorem \ref{theorem of dimension}. 

\begin{thm}\label{general edge length smoothing}  
Let $(X_0,\bm n)$ be as in Theorem \ref{theorem of dimension}.
%satisfy the conditions of \cite[Corollary 5.2]{osserman2014dimension}, namely:
%$\ \ \ \mathrm{(I)}$ there are at most three edges of $G$ connecting any given pair of vertices;
%$\ \mathrm{(II)}$ there exists $d$ such that  for any adjacent vertices $v, v'$ of $\Gamma$, connected by edges $(e_i)_i$, and any integers $(x_i)_i$ with $\sum_i x_i\bm n(e_i)=0$, if there is a unique $j$ with $x_j>0$ then we have $\sum_i\lfloor x_j\bm n(e_j)/\bm n(e_i)\rfloor>d$;
%$\mathrm{(III)}$ each marked component of $X$ is strongly Brill-Noether general.
Then for any limit linear series $(\mathcal D,\mathcal H)$ on $\mathfrak C_{X_0,\bm n}$ of multidegree $w_0$ with $\deg(\mathcal D)\leq d'$ the following are equivalent:

$\mathrm{(1)}$ $(\mathcal D,\mathcal H)$ is smoothable;

$\mathrm{(2)}$ $(\mathcal D,\mathcal H)$ satisfies the weak glueing condition;

$\mathrm{(3)}$ $(\mathcal D,\mathcal H)$ is the image under $\mathfrak F_{(w_v)_v}$ of a pre-limit linear series on $(X_0,\bm n)$ that lifts to a limit linear series. 

Moreover we have $(1)\Leftrightarrow(2)$ when $\mathcal D$ is only required to be rational.
%is smoothable if and only if it satisfies the weak glueing condition. 
\end{thm}

\begin{lem}\label{weak glueing smoothing} 
Let $(X_0,\bm n)$ be as in Theorem \ref{general edge length smoothing} and $d\leq d'$. Then a pre-limit linear series $(\mathscr L_v, V_v)_v$ with respect to $(w_v)_v$ lifts to a limit linear series if and only if $(\mathscr L_v, V_v)_v$ satisfies the weak glueing condition.
\end{lem}
\begin{proof}
By proposition \ref{lls weak} it remains to prove the ``if" part. Take any line bundle $\mathscr L$ on $\widetilde X_0$ of multidegree $w_0$ such that $\mathscr L^v=\mathscr L_v$. For any adjacent vertices $v,v'\in V(G)$ connected by $e\in E(\overline G)$, according to the proof of Theorem \ref{theorem of dimension} we have either one of the following:

$(a)$ there is at most one $j$ such that $\deg D_{j+1}^{e,v}-\deg D_j^{e,v}=3$ while $\deg D_{j'+1}^{e,v}-\deg D_{j'}^{e,v}\leq 1$ for $j'\neq j$ and $0\leq j'\leq b_{v,v'}$;

$(b)$ there are at most two numbers $\{j_i\}_{1\leq i\leq m}$ where $0\leq j_i\leq b_{v,v'}$ and $m\leq 2$ such that $\deg D_{j_i}^{e,v}-\deg D_{j_i-1}^{e,v}=2$, and when $m=2$ the supports of $D_{j_1+1}^{e,v} -D_{j_1}^{e,v}$ and $D_{j_2+1}^{e,v} -D_{j_2}^{e,v}$ only overlap at one point.

For any given enriched structure $(\mathscr O_v)_v$ we check that condition (2) of Definition \ref{limit linear series on curves} will be satisfied after adjusting the glueing data of $\mathscr L$ along the edges over $e$, which has dimension $c=\#\{e'|e'\in E(G)\ \mathrm{lies\ over\ } e\}-1$. 
The nontrivial cases is when $\deg D_{j+1}^{e,v}-\deg D_j^{e,v}=3$ and $g_j\leq 2$, which impose codimension at most $2$ (when $g_j=2$ we can take basis of $V(-D_{j}^{e,v})/ V(-D_{j+1}^{e,v})$ that vanish at different points, each of which impose codimension at most one); or when $\deg D_{j+1}^{e,v}-\deg D_j^{e,v}=2$ and $g_j=1$, which impose codimension at most $1$. Now the first case happens only if $c= 2$ and the second case happens when $m=1$ and $c\geq 1$ or $m=2$ and $c= 2$.
\end{proof}

\begin{proof}[Proof of Theorem \ref{general edge length smoothing}]
We have $\mathrm{(1)}\Rightarrow \mathrm{(2)}$ according to Theorem \ref{smoothness implies weak glueing}. Lemma \ref{weak glueing smoothing} and Theorem \ref{theorem of dimension} and the smoothing theorem (Theorem \ref{a smoothing theorem}) together shows $\mathrm{(2)}\Rightarrow \mathrm{(3)}$. And $\mathrm{(3)}\Rightarrow\mathrm{(1)}$ is Lemma \ref{compatibleness between smoothing}.

\end{proof}

When $\rho=0$, the weak glueing condition in Theorem \ref{general edge length smoothing} is trivial, since according to the proof of Theorem \ref{theorem of dimension} the space of pre-limit $\mathfrak g^r_d$s with extra vanishing has dimension strictly less that $\rho$, which is empty. As a result we have:

\begin{cor}\label{rho equals zero}
Let $(X_0,\bm n)$ be as in Theorem \ref{general edge length smoothing}, and $(\mathcal D,\mathcal H)$ a limit $\mathfrak g^r_d$ of multidegree $w_0$ such that $d\leq d'$ and $\rho=0$, then $(\mathcal D,\mathcal H)$ is smoothable.
\end{cor}
%\begin{proof}
%In this case the space of pre-limit linear series of rank $r$ has dimension less than or equal to $\rho$, since in the proof of Theorem \ref{theorem of dimension} we see that the glueing conditions ((II) of Definition \ref{limit linear series on curves}) are imposed on the glueing data over the nodes. In particular the subspace of pre-limit linear series with extra vanishing (which is of positive codimension) is empty because the components $Z_v$ are strongly Brill-Noether general. 
%\end{proof}

\begin{ex} \label{example}
We give an example of a limit linear series on a metrized complex that does not satisfy the weak glueing condition, hence not smoothable. Let $X_0$ be a nodal curve obtained by glueing two copies of $\mathbb P^1_\kappa$ along two points. Let $V(G)=\{v,v'\}$ and $E(G)=\{e_1,e_2\}$. Let $\bm n$ be the chain structure such that $\bm n(e_1)=2$ and $\bm n(e_2)=1.$ Let $\mathcal A_{v}=\{P,Q\}$ and $A_{v'}=\{P',Q'\}$. The metrized complex is as below, where $\widetilde P$ is the midpoint of $e_1$.

$$
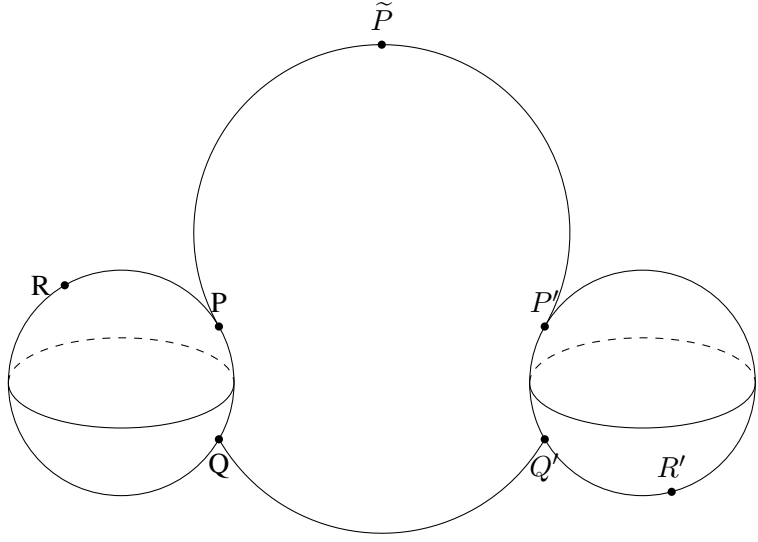
\begin{tikzpicture}
\draw (0,0) circle (1.5cm);
\draw (-1.5,0) arc (180:360:1.5cm and 0.6cm);
\draw[dashed] (-1.5,0) arc (180:0:1.5cm and 0.6cm);
\draw (1.3,-0.75) arc (210:330:2.5cm and 2.5cm);
\draw (1.3,0.75) arc (210:-30:2.5cm and 2.5cm);
\draw (6.93,0) circle (1.5cm);
\draw (5.43,0) arc (180:360:1.5cm and 0.6cm);
\draw[dashed] (5.43,0) arc (180:0:1.5cm and 0.6cm);
\draw (1*360/12: 1.5cm) node[circle, fill=black, scale=0.3, label=above:{P}]{};
\draw (4*360/12: 1.5cm) node[circle, fill=black, scale=0.3, label=left:{R}]{};
\draw (11*360/12: 1.5cm) node[circle, fill=black, scale=0.3, label=below:{Q}]{};
\draw (5*360/12: 1.5cm)+(6.93,0) node[circle, fill=black, scale=0.3, label=above:{$P^\prime$}]{};
\draw (9.5*360/12: 1.5cm)+(6.93,0) node[circle, fill=black, scale=0.3, label=above:{$R^\prime$}]{};
\draw (3*360/12: 2.5cm)+(3.464,2) node[circle, fill=black, scale=0.3, label=above:{$\widetilde P$}]{};
\draw (7*360/12: 1.5cm)+(6.93,0) node[circle, fill=black, scale=0.3, label=below:{$Q^\prime$}]{};
\end{tikzpicture} 
$$

Consider $(\mathcal D,\mathcal H)$ such that $\mathcal D_{v}=2R$ and $\mathcal D_\Gamma=2v$ and $\mathcal D_{v'}=0$, and $H_{v}=\langle f_0,f_1\rangle$ and $H_{v'}=\langle f_0',f_1'\rangle$ where $\mathrm{div}(f_0)=P-R$, $\mathrm{div}(f_1)=P+Q-2R$, $\mathrm{div}(f_0')=R'-Q'$ and $\mathrm{div}(f_1')=0$ (namely $f_1'$ is a nonzero constant function) for some $R\in Z_v\backslash \{P,Q\}$ and $R'\in Z_{v'}\backslash \{P',Q'\}$. Choose admissible multidegrees $w_0$ and $(w_v)_v$ such that $D_{w_0}=D_{w_v}=2v$ and $D_{w_{v'}}=\widetilde P+v'$. Straightforward calculation shows that $D_v=D^v_{w_0,w_v}=0$ and $D_{v'}=D^{v'}_{w_0,w_{v'}}=Q'$. It follows that $b_{v,v'}=1$, and $f_0,f_1\in H^0(Z_v,\mathscr O_{Z_v}(\mathcal D_v+D_v))$ and $f_0',f_1'\in H^0(Z_{v'},\mathscr O_{Z_{v'}}(\mathcal D_{v'}+D_{v'}))$. The tuple  
$$((\mathscr O_{Z_v}(\mathcal D_v+D_v),H_v),(\mathscr O_{Z_{v'}}(\mathcal D_{v'}+D_{v'}),H_{v'}))=((\mathscr O_{Z_v}(2R),H_v),(\mathscr O_{Z_{v'}}(Q'),H_{v'})$$
is a pre-limit linear series on $(X_0,\bm n)$ with respect to $(w_v,w_{v'})$ (by the calculation below) that maps to $(\mathcal D,\mathcal H)$ under $\mathfrak F$. Hence $(\mathcal D,\mathcal H)$ is a limit linear series on $\mathfrak C_{X_0,\bm n}$.

Now we calculate the twisting divisors:
$$D_i^{e,v}=0,P+Q,P+2Q,2P+3Q,\cdots$$ 
$$D_i^{e,v'}=0,Q',P'+2Q',P'+3Q',\cdots$$

We have $\mathrm{div}^0(f_0)|_{\{P,Q\}}=P$ and $\mathrm{div}^0(f_1)|_{\{P,Q\}}=P+Q$ and $\mathrm{div}^0(f'_0)|_{\{P',Q'\}}=0$ and $\mathrm{div}^0(f_1')|_{\{P',Q'\}}=Q'$, which satisfies the condition for pre-limit linear series. It is easy to check that weak glueing condition is not satisfied for $j=0$, since $\mathrm{div}^0(f_0)|_{\{P,Q\}}-D_0^{e,v}=P$ but $\mathrm{div}^0(f_1')|_{\{P',Q'\}}-D_1^{e,v'}=0$. Thus $(\mathcal D,\mathcal H)$ is not smoothable. 
\end{ex} 
\ 

Note that the chain structure $\bm n$ and $X_0$ in the above example satisfies the conditions in Theorem \ref{theorem of dimension}. One can also construct non-smoothable limit linear series on $\mathfrak C_{X_0,\bm n}$ with $\bm n(e_1)=\bm n(e_2)=1$.

 \section{Lifting divisors on metric graphs} 
Let $ X'\rightarrow \mathrm{Spec}(R)$ be a regular smoothing family with generic fiber $X$. Let $G'$ be the dual graph with associated metric graph $\Gamma'$ where each edge is assigned length 1. Matthew Baker \cite{baker2008specialization} constructed a specialization map $\tau\colon \mathrm{Div}(X_{\overline K})\rightarrow \mathrm{Div}_{\mathbb Q}(\Gamma')$, where $\overline K$ denotes the algebraic closure of $K$ and $\mathrm{Div}_{\mathbb Q}(\Gamma')$ is the set of rational divisors on $\Gamma'$, with the property that $r(D)\leq r(\tau(D))$ for all $D\in \mathrm{Div}(X_{\overline K})$. The question of whether  a divisor $D\in \mathrm{Div}_{\mathbb Q}(\Gamma')$ of rank $r(D)$ lifts to a divisor in $X_{\overline K}$ of the same rank is not completely solved. A survey for partial results can be find in \cite{baker2016degeneration}. The case when $\Gamma'$ is a chain of loops with generic edge length is proved liftable in \cite{cartwright2014lifting} by a thoroughly examination of divisors on $\Gamma'$. In this section we provide a different approach of the lifting problem, via the smoothing properties of (pre-)limit linear series on curves with chain structure, for an enlarged scale of $\Gamma'$ and certain divisors on $\Gamma'$. Note that we do not require $\kappa$ to be algebraically closed, and both Definition \ref{pre-limit linear series} and Definition \ref{weak glueing condition} remain valid as well as the conclusions we use in this section (Lemma \ref{weak glueing smoothing}, Theorem \ref{a smoothing theorem} and Theorem \ref{theorem of dimension}).

Let $X_0,G,\bm n,w_0,(w_v)_v, \Gamma,\widetilde X_0$ and $\widetilde G$ be as in Notation \ref{notation}. 
 Let $(w_v^{\mathrm{red}})_v$ be as in Theorem \ref{lifting theorem}. 
\begin{thm}\label{lifting divisors chain of double loops}
Let $X$ be a smooth curve of genus $g$ over $K$. Suppose there is a regular smoothing family $ X'\rightarrow \mathrm{Spec}(R)$ with generic fiber $X$ and special fiber $\widetilde X_0$. Suppose further that $\widetilde X_0$ only has rational components, and $(X_0,\bm n)$ is as in Theorem \ref{theorem of dimension} with $d'=\min\{2g-2,g+1\}$ and that $\overline G$ is a chain. Then every rational divisor class $D$ on $\Gamma$ such that $r(D)\leq 1$ lifts to a divisor class of the same rank on $X_{\overline K}$.
\end{thm}
\begin{proof}
After a finite base field exchange we may assume that $D$ is a integral and edge-reduced divisor of multidegree $\bm w_0$. We show that $D$ can be lifted to a pre-limit linear series on $(X_0,\bm n)$ with respect to $(w_v^{\mathrm{red}})_v$ which lifts to a limit linear series with respect to the induced enriched structure on $\widetilde X_0$. This limit linear series is smoothable, and the corresponding $\mathfrak g^r_d$ on $X$ has underlying divisor specializing to $D$ (up to linear equivalence).

Let $v_0, v_1,...,v_m$ be the vertices of $G$ as in the following graph and $e_i$ be the edge of $\overline G$ that connects $v_{i-1}$ and $v_i$. Let $e_i^{1},...,e_i^{a_i}$ be the edges in $G$ over $e_i$ where $1\leq a_i\leq 3$. We may assume for simplicity that $w_0=w_{v_0}^{\mathrm{red}}$ and hence it is reasonable to denote $w_i=w_{v_i}^{\mathrm{red}}$.
 
$$
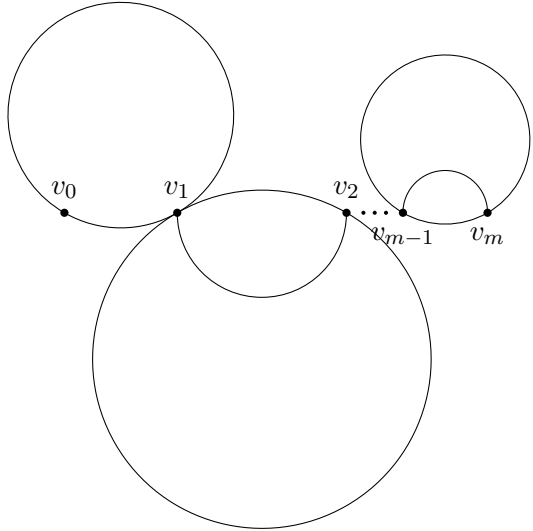
\begin{tikzpicture}[scale=0.75]
\draw (-1,0) arc (240:-60:2cm and 2cm);
\draw (-1,0) arc (-120:-60:2cm and 2cm);
\draw (1,0) arc (120:60:3cm and 3cm);
\draw (1,0) arc (-240:60:3cm and 3cm);
\draw (1,0) arc (-180:0:1.5cm and 1.5cm);
\draw (5,0) arc (240:-60:1.5cm and 1.5cm);
\draw (5,0) arc (-120:-60:1.5cm and 1.5cm);
\draw (5,0) arc (180:0:0.75cm and 0.75cm);
\draw (4.3,0) node[circle, fill=black, scale=0.15]{};
\draw (4.5,0)node[circle, fill=black, scale=0.15]{};
\draw (4.7,0)node[circle, fill=black, scale=0.15]{};
\draw (-1,0) node[circle, fill=black, scale=0.3, label=above:{$v_0$}]{};
\draw (1,0) node[circle, fill=black, scale=0.3, label=above:{$v_1$}]{};
\draw (4,0) node[circle, fill=black, scale=0.3, label=above:{$v_2$}]{};
\draw (5,0) node[circle, fill=black, scale=0.3, label=below:{$v_{m-1}$}]{};
\draw (6.5,0) node[circle, fill=black, scale=0.3, label=below:{$v_m$}]{};
\end{tikzpicture}
$$

Note that
$b_{v_{i-1},v_i}$ is is the maximal number $n$ such that the admissible multidegree obtained from $w_{i-1}$ (resp. $w_i$) by twisting $n$ times at $(e_{i},v_{i-1})$ (resp. $(e_{i},v_i)$) is nonnegative at all $v\in V(G)$. Let $Z_0,...,Z_m$ be the (rational) components of $X_0$ corresponding to $v_0,...,v_m$. Let $d_i=\deg(D_{w_i}|_{v_i})$. Since $\tau$ is surjective (\cite[\S 2.3]{baker2008specialization}) we may assume $r(D)=1$. Denote the the pre-limit linear series that we want to construct by $(\mathscr L_i,V_i)_{i}=(\mathscr O_{Z_i}(\mathcal D_i),\mathrm{span}(f_i^0,f_i^1))$. 

%First assume $r=0$. We take $\mathcal D_i=P_i^1+\cdots+P_i^{d_i}$ with $P_i^j\in Z_i\backslash( \mathcal A_{e_i}^{v_i}\cup  \mathcal A_{e_{i+1}}^{v_i})$ and $f_i^0=1$ for all $i$. Take rational functions $\gamma_{i+1}$ on $\Gamma$ such that 
%$D_i+\mathrm{div}(\gamma_{i+1})=D_{i+1}$ and take $f_i^0$ such that $\mathrm{ord}_{f_i^0}(P_{e^j_{i+1}}^{v_i})=-\mathrm{slp}_{e_{i+1}^j,v_i}(\gamma_i)$, or equivalently that $\mathrm{ord}_{f_i^0}(P_{e^j_{i+1}}^{v_i})=\mathrm{ord}_{D_{b_{v_i,v_{i+1}}}^{e_{i+1},v_i}}(x^{v_i}_{e_{i+1}^j})$, and $f_i^0$ has no zeros or poles at $\mathcal A_{e_{i-1}}^{v_i}$.

First take $\mathcal D_0=P_0^1+\cdots+P_0^{d_0}$ with $P_0^j\in Z_0\backslash \mathcal A_{e_1}^{v_0}$. If $b_{v_0,v_1}>0$,  let $f_0^0=1$ and take $f_0^1$ such that
%take rational function $\gamma$ on $\Gamma$ such that $D_0+\mathrm{div}(\gamma)=D_1$, then let $f_0^0=1$ and take $f_0^1$ such that $\mathrm{ord}_{f_0^1}(P_{e^j_1}^{v_0})=-\mathrm{slp}_{e_1^j,v_0}(\gamma)$, 
 that $\mathrm{ord}_{P_{e^j_1}^{v_0}}(f_0^1 )=\mathrm{ord}_{P^{v_0}_{e_1^j}}(D_{b_{v_0,v_1}}^{e_1,v_0})$.
 If $b_{v_0,v_1}=0$ let $f_0^0=1$ and take $f_0^1$ such that $f_0^1$ has no zeros or poles at $\{P_{e^j_1}^{v_0}\}_{j}$ and that $\Pi_{j\neq k}(f_0^1(P_{e^j_1}^{v_0})-f_0^1(P_{e^k_1}^{v_0}))\neq 0$. Next assume $(\mathscr L_{i-1}, V_{i-1})$ is given. Set $f_i^0=1$.
% and take $\mathscr L_{i}=\mathcal D_{i}$ where $\mathcal D_{i}=P_{i}^1+\cdots+P_{i}^{d_{i}}$ with $P_{i}^j\in Z_j$ general and $f_i^0=1$..

If $b_{v_{i-1},v_i}=b_{v_i,v_{i+1}}=0$ we take $\mathcal D_{i}=P_{i}^1+\cdots+P_{i}^{d_{i}}$ with $P_{i}^j\in Z_i\backslash(\mathcal A_{e_i}^{v_i}\cup \mathcal A_{e_{i+1}}^{v_i})$ general.
%we take $\mathscr L_{i}=\mathscr O(\mathcal D_{i})$ where $\mathcal D_{i}=P_{i}^1+\cdots+P_{i}^{d_{i}}$ with $P_{i}^j\in Z_j\backslash(\mathcal A_{e_i}^{v_i}\cup \mathcal A_{e_{i+1}}^{v_i})$ and
Take $f_i^1$ such that $f_i^1$ has no zeros or poles at $\mathcal A_{e_i}^{v_i}\cup \mathcal A_{e_{i+1}}^{v_i}$
%$\{P_{e^j_{i}}^{v_i}\}_{j}\cup \{P_{e^j_{i+1}}^{v_i}\}_{j}$ 
and that
$$\Pi_{j\neq k}(f_i^1(P_{e^j_{i}}^{v_i})-f_i^1(P_{e^k_{i}}^{v_i}))\cdot \Pi_{j\neq k}(f_i^1(P_{e^j_{i+1}}^{v_i})-f_i^1(P_{e^k_{i+1}}^{v_i}))\neq 0.$$ 

If $b_{v_i,v_{i+1}}=0$ and $b_{{v_{i-1},v_i}}> 0$ we take $D_i$ as above.
%rational function $\gamma_i$ on $\Gamma$ such that $D_{i-1}+\mathrm{div}(\gamma_i)=D_i$ and $\mathscr L_i=\mathscr O(\mathcal D_i)$ where $\mathcal D_i=\sum_j \mathrm{slp}_{e^j_i,v_i}(\gamma_i)x_{e^j_i}^{v_i}+\mathcal D'$ with any effective $\mathcal D'$ such that the degree of $\mathcal D_i$ is $d_i$. Take 
Take $f_i^1$ such that $f_i^1$ has no zeros or poles at $\mathcal A_{e_{i+1}}^{v_i}$ and that $\Pi_{j\neq k}(f_i^1(P_{e^j_{i+1}}^{v_i})-f_i^1(P_{e^k_{i+1}}^{v_i}))\neq 0$ and that $\mathrm{ord}_{P_{e^j_i}^{v_i}}(f_i^1)=\mathrm{ord}_{P_{e^j_i}^{v_i}}(D_{b_{v_{i-1},v_i}}^{e_i,v_{i}})$.

If $b_{v_i,v_{i+1}}>0$ and $b_{{v_{i-1},v_i}}=0$ we take $ D_i$ as above. %we take rational function $\gamma_{i+1}$ on $\Gamma$ such that $D_{i}+\mathrm{div}(\gamma_i)=D_{i+1}$ and $\mathscr L_i=\mathscr O(\mathcal D_i)$ where $\mathcal D_i=P_i^1+\cdots+P_i^{d_i}$ with $P_i^j$ general (out of $\mathcal A_{e_i}^{v_i}\cup \mathcal A_{e_{i+1}}^{v_i}$) 
Take and $f_i^1$ such that $f_i^1$ has no zeros or poles at $\mathcal A_{e_i}^{v_i}$ and that $\Pi_{j\neq k}(f_i^1(P_{e^j_{i}}^{v_i})-f_i^1(P_{e^k_{i}}^{v_i}))\neq 0$ and that $\mathrm{ord}_{P_{e^j_{i+1}}^{v_i}}(f_i^1)=\mathrm{ord}_{P_{e^j_{i+1}}^{v_i}}(D_{b_{v_{i},v_{i+1}}}^{e_{i+1},v_{i}}).$

%$\mathcal D_i=-\sum_j \mathrm{slp}_{e^j_{i+1},v_i}(\gamma_{i+1})P_{e^j_{i+1}}^{v_i}+\mathcal D'$ with $\mathcal D'$ as above. Take $f_i^1$ such that $f_i^1$ has no zeros or poles at $\mathcal A_{e_i}^{v_i}$ and that $\Pi_{j\neq k}(f_i^1(P_{e^j_{i}}^{v_i})-f_i^1(P_{e^k_{i}}^{v_i}))\neq 0$ and that $\mathrm{ord}_{f_i^1}(P_{e^j_{i+1}}^{v_i})=\mathrm{slp}_{e^j_{i+1},v_i}(\gamma_{i+1}).$

If $b_{v_i,v_{i+1}}>0$ and $b_{{v_{i-1},v_i}}>0$ 
we take $\mathcal D_i=\sum_j \mathrm{ord}_{P_{e^j_i}^{v_i}}(D^{e_i,v_i}_{b_{v_{i-1},v_i}})P_{e^j_i}^{v_i}+\mathcal D'=D^{e_i,v_i}_{b_{v_{i-1},v_i}}+\mathcal D'$ with $\mathcal D'$ an effective divisor supported on $Z_i\backslash(\mathcal A_{e_i}^{v_i}\cup \mathcal A_{e_{i+1}}^{v_i})$ such that $\deg \mathcal D_i=d_i$. Take $f_i^1$ such that $\mathrm{ord}_{P_{e^j_i}^{v_i}}(f_i^1)=-\mathrm{ord}_{P_{e^j_i}^{v_i}}(D^{e_i,v_i}_{b_{v_{i-1},v_i}})$ and $\mathrm{ord}_{P_{e^j_{i+1}}^{v_i}}(f_i^1)=\mathrm{ord}_{P_{e^j_{i+1}}^{v_i}}(D^{e_{i+1},v_i}_{b_{v_{i},v_{i+1}}}).$

Now let $1\leq i\leq g$. If $b_{v_{i-1},v_{i}}>0$ then $\mathrm{ord}_{D_\bullet^{e_{i},v_{i-1}}}f_{i-1}^0=\deg D_0^{e_{i},v_{i-1}}$ and $\mathrm{ord}_{D_\bullet^{e_{i},v_{i-1}}}f_{i-1}^1=\deg D_{b_{v_{i-1},v_{i}}}^{e_{i},v_{i-1}}$; we also have
$\mathrm{ord}_{D_\bullet^{e_{i},v_{i}}}f_i^0=\deg D_{b_{v_{i-1},v_{i}}}^{e_{i},v_{i}}$ and $\mathrm{ord}_{D_\bullet^{e_{i},v_{i}}}f_i^1=\deg D_{0}^{e_{i},v_{i}}$ when $b_{v_i,v_{i+1}}>0$ while $\mathrm{ord}_{D_\bullet^{e_{i},v_{i}}}f_i^1=\deg D_{b_{v_{i-1},v_{i}}}^{e_{i},v_{i}}$ and $\mathrm{ord}_{D_\bullet^{e_{i},v_{i}}}f_i^0=\deg D_{0}^{e_{i},v_{i}}$ when $b_{v_i,v_{i+1}}=0$.
 Note that $0$ and $b_{v_{i-1},v_i}$ are critical for both $D_\bullet^{e_{i},v_{i-1}}$ and $D_\bullet^{e_{i},v_{i}}$, hence $(\mathscr L_i,V_i)_i$ is a pre-limit linear series with respect to $(w_i)_{v_i}$. Moreover, one checks easily that $(\mathscr L_i, V_i)_{v_i}$ satisfies the weak glueing condition. 

%On the other hand, if $b_{v_{i-1},v_{i}}>0$ then by construction $\overline f_{i-1}^0$ is in the dense torus orbit of the space $\mathscr L_{i-1}(-D_0^{e_{i},v_{i-1}})/\mathscr L_{i-1}(-D_1^{e_{i},v_{i-1}})$, and similarly for $\overline f_{i-1}^1$ and $\overline f_{i}^0$ and $\overline f_{i}^1$. Using the same argument as in Lemma \ref{weak glueing smoothing}, we found that the generality of edge lengths ensures that the case that $\deg(D_{j+1}^{e_i,v_{k}}-D_{j}^{e_i,v_{k}})>1$ where $i-1\leq k\leq i$ and $0\leq j\leq b_{v_{i-1},v_{i}}$ can only happen limited times and in specific situations, hence we have enough glueing data at $e_i$ such that $f_{i-1}^{j}$ is glued to $f_{i}^{1-j}$ under any enriched structure. If $b_{v_{i-1},v_{i}}=0$ we only need to consider the case that $\deg(D_{1}^{e_i,v_{k}}-D_{0}^{e_i,v_{k}})=3$, where the glueing condition is nontrivial. Choose $s_{k}^j\in V_k$ for $j=0,1$ and $k=i-1,i$ such that $s_k^j$ only vanish at $P_{e^{j+1}_i}^{v_k}$ and glue $s^j_{i-1}$ with $s_i^{1-j}$. Thus $(\mathscr L_i,V_i)_i$ is actually a limit linear series of rank $r$. 

Now since $r(D)= 1$, the Riemann-Roch theorem for graphs \cite[Proposition 3.1]{gathmann2008riemann} shows that $\deg D\leq g+1$. If $\deg D\leq 2g-2$, Lemma \ref{weak glueing smoothing} implies that $(\mathscr L_i, V_i)_{v_i}$ lifts to a limit $\mathfrak g^1_d$ on $(X_0,\bm n)$, which is smoothable by Theorem \ref{a smoothing theorem} and Theorem \ref{theorem of dimension}. Thus $D$ can be lifted to a rank-one divisor on $X$. If $\deg D> 2g-2$ we can take a divisor $\mathcal D\in \mathrm{Div}(X)$ that specializes to $D$ and the Riemann-Roch theorems (on curves and graphs) would imply that $r(\mathcal D)=r(D)$.
%hence Theorem \ref{a smoothing theorem} and Theorem \ref{theorem of dimension} implies that $(\mathscr L_i,V_i)_i$ can be lifted to a limit linear series of rank $r$ on $X$ if $\deg D\leq 2g-2$. On the other hand, since the conditions in Corollary \ref{general edge length smoothing} is satisfied automatically for $d\leq 1$, same reasoning shows that $D$ is liftable to a divisor on $X$ of rank $r$ if $\deg D\leq 1$. We thus conclude that the specialization map $\mathrm{Div}(X)\rightarrow \mathrm{Div}(\Gamma)$ is surjective onto the set of the equivalent classes of rational divisors. If $\deg D> 2g-2$ we can take a divisor $\mathcal D\in \mathrm{Div}(X)$ that specializes to $D$ and the Riemann-Roch theorems (on curves and graphs) would imply that $r(\mathcal D)=r$.
\end{proof}

\begin{rem}
If we set $d'=2g-2$ in Theorem \ref{theorem of dimension}, then as described in \cite[Remark 5.4]{osserman2014dimension}, this recovers precisely the general curves considered in \cite{cartwright2014lifting}. Thus the Theorem above provides an alternate proof for \cite[Theorem 1.1]{cartwright2014lifting} for divisors of rank less than or equal to one (over a complete discrete valued field). Later we will give another proof of loc.cit. for vertex avoiding divisors.
\end{rem}

Combining with the tropical Riemann-Roch theorem, we can prove Theorem \ref{lifting divisors chain of double loops} without requiring $r(D)\leq 1$ for $X$ with small genus: 

\begin{cor}\label{lifting small genus}
In Theorem \ref{lifting divisors chain of double loops}, suppose further that $g\leq 5$, then any rational divisor class $D$ on $\Gamma$ can be lifted to a divisor class on $X_{\overline K}$ of the same rank as $D$.
\end{cor}
\begin{proof}
We may assume that $r(D)\geq 2$ and, by Riemann-Roch theorem, that $r(K_\Gamma-D)\geq 2$. We have $r(K_\Gamma)=g-1\geq r(D)+r(K_\Gamma-D)\geq 4$, hence it remains to consider the case $g=5$ and $r(D)=r(K_\Gamma-D)=2$. Now the tropical Clifford's theorem \cite[Theorem 1]{facchini2010tropical} shows that $\deg D=4$ and $D=2D'$ where $r(D')=1$. It follows from Theorem \ref{lifting divisors chain of double loops} that $D'$ is liftable to a divisor of the same rank, hence so is $D$.
\end{proof}

We next show the lifting of vertex avoiding divisors on a generic chain of loops. Since only one proposition is needed in the proof, we refer to \cite{cartwright2014lifting} for the definition of vertex avoiding divisors. The idea of proof is from Sam Payne at a workshop. 
\begin{thm}\label{lifting vertex avoiding divisors}
In Theorem \ref{lifting divisors chain of double loops} let $d'=2g-2$ instead. Suppose further that every pair of adjacent vertices is connected by at most two edges. Then every rational vertex avoiding divisor $D$ on $\Gamma$ lifts to a divisor class  on $X_{\overline K}$ of the same rank as $D$.
\end{thm}

\begin{proof}
We use the same strategy as in the proof of Theorem \ref{lifting divisors chain of double loops}. Suppose $D$ is $v_0$-reduced. Let $v_i,e_i, e_i^j,w_i$ and $d_i$ be as in loc.cit,. According to \cite[Proposition 2.4]{cartwright2014lifting}, for each $0\leq j\leq r$ there exists a unique divisor $D_j$ linearly equivalent to $D$ such that $D_j-jv_0-(r-j)v_m$ is effective, which is obtained by taking a pile of $d_0-j$ chips from $v_0$ and moving to the right (see the proof of \cite[Proposition 6.3]{jensen2014tropical}), where $d_0$ is the coefficient of $D$ at $v_0$ as in Theorem \ref{lifting divisors chain of double loops}. 

Take rational functions $f_\Gamma^j$ on $\Gamma$ such that $D_j=D+\mathrm{div}(f_\Gamma^j)$. For $0\leq i\leq m$ and $0\leq j\leq r$ let $D_i^j=\sum_k\mathrm{slp}_{e^k_i,v_i}(f_\Gamma^j)P_{e_{i}^k}^{v_i}$ and $E_i^j=-\sum_k\mathrm{slp}_{e^k_{i+1},v_i}(f_\Gamma^j)P_{e_{i+1}^k}^{v_i}$ be divisors on $Z_i$. Hence $D_i^j$ represents the chips we get when we try to move the pile of $r-j$ chips of $v_0$ from $v_{i-1}$ to $v_i$, while $E_i^j$ represents the chips we lost when move the chips from $v_i$ to $v_{i+1}$ (note that the size of the pile of chips may change as it moves).

For each $i$, by construction we have $D_i^j=D^{e_{i},v_i}_{b_{v_{i-1},v_{i}}}-D^{e_i,v_i}_{m_{i}^j}$ for some $m_{i}^j$ critical and $E_i^j=D^{e_{i+1},v_i}_{n_{i}^j}$ for some $n_i^j$ critical such that $n_i^j+m_{i+1}^j=b_{v_i,v_{i+1}}$. The uniqueness of $D_j$ shows that the $D_i^j$s are distinct for fixed $i$ and so are the $E_i^j$s. Now we take $\mathcal D_i=D^{e_{i},v_i}_{b_{v_{i-1},v_{i}}}+\mathcal D_i'$ 
% $\mathcal D_i\subset Z_i\backslash (\mathcal A_{e_i}^{v_i}\cup \mathcal A_{e_{i+1}}^{v_i})$ 
to be a divisor of degree $d_i$ on $Z_{v_i}$ where $\mathcal D_i'\subset Z_{v_i}\backslash \mathcal A_{v_i}$ is effective and $f_i^j\in\mathscr O_{Z_{v_i}}(\mathcal D_i)$ such that $\mathrm{ord}_{P_{e_{i}^k}^{v_i}}(f_i^j)=-\mathrm{ord}_{P_{e_{i}^k}^{v_i}}(D_{i}^{j})$ and $\mathrm{ord}_{P_{e_{i+1}^k}^{v_i}}(f_i^j)=\mathrm{ord}_{P_{e_{i+1}^k}^{v_i}}(E_i^j)$ for all possible $k$. Note that the existence of $f_i^j$ follows from the effectiveness of $D_{j}$. It is easy to check that the tuple $(\mathscr L_i,V_i)_{v_i}=(\mathscr O_{Z_{v_i}}(\mathcal D_i),\mathrm{span}(\{f_i^j\}_j))_{v_i}$ is a pre-limit linear series on $X_0$ which satisfies the weak glueing condition, hence it lifts to a limit linear series on $(X_0,\bm n)$. As a result $D$ can be lift to a divisor of rank $r(D)$ on $X$.
\end{proof}

\bibliographystyle{amsalpha}
\bibliography{1}
\end{document}